\documentclass[12pt,reqno]{amsart}

\usepackage{amsmath,amssymb,amsfonts,amsthm,latexsym,graphicx,multirow,enumerate,mathrsfs,pgfplots,geometry,multirow,booktabs}
\usetikzlibrary{arrows,calc,arrows.meta}
\usepackage{todonotes}
\pgfplotsset{compat=1.15}

\usepackage[colorlinks,linkcolor=blue]{hyperref}

\newtheorem{theorem}{Theorem}[section]
\newtheorem{lemma}[theorem]{Lemma}
\newtheorem{proposition}[theorem]{Proposition}
\newtheorem{corollary}[theorem]{Corollary}

\theoremstyle{definition}

\theoremstyle{definition}
\newtheorem{definition}[theorem]{Definition}

\theoremstyle{remark}
\newtheorem{example}[theorem]{Example}
\newtheorem*{remark}{Remark}

\newcommand\lcm{\operatorname{lcm}}
\newcommand\A{\operatorname{A}}
\newcommand\PSL{{\rm PSL}}
\newcommand\PGL{\operatorname{PGL}}
\newcommand\Aut{\operatorname{Aut}}
\newcommand\Out{\operatorname{Out}}
\newcommand{\GL}{\operatorname{GL}}
\newcommand{\SL}{\operatorname{SL}}

\newcommand{\Map}{\operatorname{Map}}
\newcommand{\rad}{\operatorname{rad}}
\newcommand{\notlhd}{\not\hspace*{-.5em}\lhd}

  \def\GaL{\mathrm{\Gamma L}}
\def\PGaL{\mathrm{P\Gamma L}}   
   
  \def\PSigL{\mathrm{P\Sigma L}}
\def\PGaU{\mathrm{P\Gamma U}}

\def\S{{\rm S}}
\def\Z{{\rm Z}}
\def\D{{\rm D}}
\def\O{{\rm O}}
\def\J{{\rm J}}
\def\a{\alpha}
\def\calM{\mathcal{M}}
\def\calm{\mathcal{M}}
\def\l{\langle}
\def\r{\rangle}

\def\num{{\sf Num}}

\def\P#1#2{\mathscr{P}_{#1}(#2)}

\def\rr{x} 
\def\tt{y} 

\title[Bi-rotary maps with $\chi=-p^d$]{On Bi-rotary Maps of Negative Prime Power Euler Characteristic}

\author[Chen]{Jiyong Chen}
\address{$^{1}$ School of Mathematical Sciences\\
Xiamen University\\
Xiamen, Fujian 361005\\
P. R. China}
\email{chenjy1988@xmu.edu.cn}

\author[Ding]{Zhaochen Ding}
\address{$^{2}$ Department of Mathematics,\\
University of Auckland,\\
Private Bag 92019, Auckland 1142, New Zealand}
\email{dzha470@aucklanduni.ac.nz}

\author[Li]{Cai Heng Li}
\address{$^{3}$ SUSTech International Center for mathematics\\
Department of Mathematics\\
Southern University of Science and Technology \\
Shenzhen 518055\\
P. R. China}
\email{lich@sustech.edu.cn}



\usepackage{todonotes}

\begin{document}

\begin{abstract}

A map is bi-orientable if it admits an assignment of local orientations to its vertices such that for every edge, the local orientations at its two endpoints are opposite. Such an assignment is called a bi-orientation of the map.
A bi-orientable map is bi-rotary if its automorphism group contains an arc-regular subgroup that preserves the bi-orientation. In this paper, we characterize the automorphism group structure of bi-rotary maps whose Euler characteristic is a negative prime power.


\end{abstract}

\maketitle

\section{Introduction}


A \emph{map} is a cellular embedding of a connected graph $\Gamma$ on a (closed) surface $\mathcal{S}$.
The graph $\Gamma$ is called the \emph{underlying graph} of the map, while the surface $\mathcal{S}$ is called the \emph{supporting surface} of the map. The \emph{Euler characteristic of a map} is the Euler characteristic of the supporting surface of the map.


Let $\calM$ be a map.
For an edge $e=[\a,e,\a']$, the two faces of $\calM$ incident with $e$ is denoted by $f$ and $f'$.
For a subgroup $X\le\Aut\calM$, the map $\calM$ is called \emph{$X$-vertex-rotary} if $X$ is arc-regular on the $\calM$, and $X_\a$ is cyclic.
In this case, $X$ contains an element $\rr$ and an involution $\tt$ such that 
\[X_\a=\l \rr\r,\ X_e=\l \tt\r.\]
We call the pair $(\rr,\tt)$ a \emph{rotary pair} of $\mathcal{M}$.
Note that $X_\a$ acts regularly on $E(\a)$, the edge set incident with $\a$, and
$\tt$ fixes $e$ and interchanges $\a$ and $\a'$, and 
\[\mbox{either $\tt$ interchanges $f$ and $f'$, or $\tt$ fixes both $f$ and $f'$.}\] 
In the former case, $\calM$ is an \emph{$X$-rotary map} (also called \emph{orientably regular}).
For the latter, $\calM$ is an \emph{$X$-bi-rotary map}.
In the work cited as \cite{li2024Locally}, it is demonstrated that $X=\Aut^+(\mathcal{M})$, the orientation-preserving automorphism group of $\mathcal{M}$ if $\mathcal{M}$ is $X$-rotary, and $X=\Aut^b(\mathcal{M})$, the bi-orientation-preserving automorphism group of $\mathcal{M}$ if $\mathcal{M}$ is $X$-bi-rotary.
Let $k$ be the valency of $\mathcal{M}$ and let $m$ be the face length of $\mathcal{M}$. Then $(k,m)$ is the \emph{type} of $\mathcal{M}$.

If $\calM$ is an \emph{$X$-bi-rotary} map, then it is known that 
the face stabilizer $X_f$ is $\l \tt,\tt^\rr\r$, which has two orbits on the edges incident with $f$.
Since $X_\a$ is arc-regular, we have $k=|\rr|$ and $m=|\l\tt,\tt^\rr\r|=2\cdot|\tt^\rr\tt|=2\cdot|[\rr,\tt]|$.
Therefore, the Euler characteristic of $\calM$ is
\[|X|(\frac{1}{|\rr|}-\frac{1}{2}+\frac{1}{2\cdot|[\rr,\tt]|})=|X|(\frac{1}{k}-\frac{1}{2}+\frac{1}{m}).\]
We refer to \cite{dazevedoBirotaryMapsNegative2019,li2024Locally} for details.

Bi-rotary maps were introduced by Wilson \cite{wilson_1978} and are 
 closely related to rotary maps via the Petrie dual operator \cite{LINS1982171}, which transforms rotary maps into bi-rotary maps and vice versa while preserving the underlying graphs.
 Rotary maps and bi-rotary maps are two of the fourteen classes of edge-transitive maps defined in \cite{graver1997Locally} and represent the highest `level of symmetry' with respect to preservation of orientations or bi-orientations.
 Significant contributions have been made to rotary maps (see, for example, \cite{Asciak2023OrientableAN,Conder2016OrientablyregularMW,orireg-multiK_n,james1985Regular}). However, research on bi-rotary maps is still relatively underdeveloped compared to that on rotary maps.

 A central problem in topological graph theory is classifying `highly symmetric' maps on a given surface. It is well-known that the Euler characteristic and orientability of a surface uniquely determine the surface, so one would like to characterize or classify bi-rotary maps of a given Euler characteristic.
 Although there are many studies on the characterization of rotary maps and regular maps (maps whose full automorphism group acts transitively on flags) with a specific Euler characteristic (see, for example, \cite{dazevedo2004Classification,Conder2012ClassificationOR,CONDER20102620,Gill2013,Hua2024RegularMW}), there has been relatively little work focused on bi-rotary maps.
 Antonio and d'Azevedo \cite{bredadazevedoRegularPseudoorientedMaps2015} used computational methods to classify bi-rotary maps of Euler characteristic $\chi\ge -16$.
In 2019, d'Azevedo, Catalano, and \v{S}ir\'{a}\v{n} \cite{dazevedoBirotaryMapsNegative2019} classified bi-rotary maps of negative prime characteristic. 
Li, Praeger, and Song \cite{li2024Locally} introduced the concept of a vertex-rotary map to encompass rotary maps and bi-rotary maps, addressing them uniformly through the framework of coset map theory.

In this paper, we aim further to characterize bi-rotary maps of negative prime power Euler characteristic.

Let $\mathcal{M}$ be an $X$-bi-rotary map with a rotary pair $(x,y)$.
It is known that $X=\l x,y\r$ and $|y|=2$ (see \cite{li2024Locally}).
Conversely,
it is known that, given a group $X$ with elements $x,y\in X$ such that $X=\l x,y\r$ and $|y|=2$, there is a unique $X$-bi-rotary map (up to isomorphism) $\mathcal{M}$ such that $(x,y)$ is a rotary pair of $\mathcal{M}$ (see \cite[Section 3]{dazevedoBirotaryMapsNegative2019} or \cite[Section 4.2]{li2024Locally} for a detailed construction).
In this case, we denote the corresponding map $\calm$ by $\Map(X,x,y)$ and say that $(x,y)$ is a \emph{rotary pair} of $X$.
Therefore, our problem is equivalent to characterize a group $X$ with a rotary pair $(\rr,\tt)$ of $X$ such that
\[|X|(\frac{1}{|\rr|}-\frac{1}{2}+\frac{1}{2\cdot|[\rr,\tt]|})=-p^n\]
for some prime $p$.

Let $\calm$ be an $X$-bi-rotary map of an Euler characteristic equal to a negative prime power $-p^n$, let $G=X/\O_p(X)$, let $\rho=x\O_p(X)$, and let $\tau=y\O_p(X)$. 
Set $\bar{k}=|\rho|$ and $\bar{m}=2\cdot|[\rho,\tau]|$.
Then $k=p^\alpha\bar{k}$ and $m=p^\beta\bar{m}$ for some integers $\alpha$ and $\beta$.
Note that $\Map(G,\rho,\tau)$ is of type $(\bar{k},\bar{m})$ and is a quotient map of $\mathcal{M}$  (see \cite{dazevedoBirotaryMapsNegative2019}).

In this paper, we characterize $X$ by classifying all possible structures of $G$.
We note that a special case occurs when $p \nmid |X|$, which is examined in \cite{Li2024ArctransitiveMW}. However, determining the conditions under which the Euler characteristic is equal to $-p^n$ remains a challenge.
The main results are summarized in the following two theorems, addressing the solvable and non-solvable cases, respectively.


\begin{theorem}\label{thm:main-result-solvable}
    If $G$ is non-abelian and solvable,
    then $G=\l a\r{:}(\l b\r{\times}H)$ where $H$ is a Hall $\{2,3\}$-subgroup of $G$, and all possible $H$ and $(\bar{k},\bar{m})$ are listed in Table~\ref{table solvable}.
    Moreover, in each case, $G$ has a rotary pair $(\rho,\tau)$, and the examples of $\mathcal{M}$ corresponding to each case are constructed in Section~\ref{sec:example}.
\end{theorem}

In Table~\ref{table solvable}, all possible $H$ and type $(\bar{k},\bar{m})$ are listed. We also provide a possible choice of $(\rho,\tau)$.
The parameters $e$ and $f$ are positive integers.
The parameters $k_1=|\rho_0|$, $k_2=|b|$, and $m_2=|a|$.
Let \[\kappa=\left\{\begin{matrix}b\rho_0&\mbox{if }H\cong \Z_{k_1}\mbox{ or }\Z_{k_1}\times\Z_2;\\
bc& \mbox{if } H\cong\operatorname{D}_{2\cdot 3^e}\mbox{ and }p=2;\\
bcd&\mbox{if } H\cong\operatorname{D}_{2\cdot 3^e}\mbox{ and }k=2k_2; \\
bcd_1\mbox{ or }
bcd_1d_2&\mbox{if } H\cong\Z_{2^f}{\times}\operatorname{D}_{2{\cdot} 3^e};\\
bc&\mbox{if } H\cong\Z_2^2{:}\Z_{3^e}.
\end{matrix}\right.\]
The parameter $m'=|\operatorname{C}_{\l a\r}(\l\kappa\r)|$.
\begin{table}[ht]
\label{table solvable}

\begin{tabular}{cccc}
$H$  & $(\bar{k},\bar{m})$ &  $(\rho,\tau)$ &   Comments \\  \midrule[1pt]
$\Z_{k_1}=\l \rho_0\r$     & $(k_1k_2m_2',2m_2)$ & $(ab\rho_0,\rho_0^{k_1/2})$ & \\ \hline

$\Z_{k_1}{\times} \Z_2=\l \rho_0\r{\times }\l \tau_0\r$     & $(k_1k_2m_2',2m_2)$ & $(ab\rho_0,\tau_0)$ & $p=2$ \\ \hline

\multirow{2}{*}{$\operatorname{D}_{2{\cdot} 3^e}= \l c\r {:} \l d\r$}    & $(3^ek_2m_2',2{\cdot} 3^e m_2)$ & $(abc,d)$ &  $p=2$ \\\cline{2-4}

        & $(2k_2,2{\cdot} 3^e m_2)$ & $(abcd,d)$ & $p\ne 3$\\ \hline

\multirow{2}{*}{$\Z_{2^f}{\times}\operatorname{D}_{2{\cdot} 3^e} =\l d_1\r{\times}\l c\r{:}\l d_2\r$}    & $(2^f3^ek_2m_2',2{\cdot} 3^e m_2)$ & $(abcd_1,d_2)$ &\multirow{2}{*}{ $p=2$} \\\cline{2-3}

        & $(2^fk_2m_2',2{\cdot} 3^e m_2)$ & $(abcd_1d_2,d_2)$ & \\ \hline

$\Z_{2}^2{:} \Z_{3^e}=(\l d_1\r{\times}\l d_2\r){:}\l c\r$     & $(3^ek_2,4)$ & $(bc,d_1)$ & $p\ne 2$ and $a=1$ \\ \bottomrule[1pt]
\end{tabular}

\vspace*{2mm}

\caption{Possible structures for $H$ and standard generating pairs for $G$}
\end{table}

\begin{theorem}\label{thm:main-non-solvable}
    Denote $\rad(G), G^{(\infty)}$ by $R,D$ respectively. 
    Let $\num$ be the set consisting of powers of 2 greater than or equal to 4 and Mersenne primes, as well as Fermat primes.
    If $G$ is  non-solvable, then
    $G=(R\times D).\Z_f$ where $f\le 2$ and $D\cong\PSL(2,q)$. 
    Moreover,  one of the following holds:
    \begin{itemize}
        \item [\rm (i)] $p>2$, $f=1$,  $R$ is a cyclic group of odd order, and $q=2p^t\pm 1$ is a prime or $q=p^t$ is a power of $p$;
        \item [\rm (ii)] $p=2$, $f=1$, and $q\in {\num}$;
        \item [\rm (iii)] $p=2$, $f=2$,  $G=R\times(\PSL(2,q).\Z_2)$ and  $q\in {\num}$;
        \item [\rm (iv)] $p=2$, $f=2$, $\O_2(G/D)=1$, and  $q\in {\num}$.
    \end{itemize}
    Examples for each case are constructed in Section~\ref{sec:example}.
\end{theorem}

\begin{remark}
    \begin{enumerate}[{\rm (1)}]
        \item In Theorems~\ref{thm:main-result-solvable} and \ref{thm:main-non-solvable}, we characterize only the quotient group $X/\O_p(X)$, rather than attempting a full characterization of $X$ itself. 
        This is because $\O_p(X)$ can have an arbitrarily complex structure. For instance, see the examples constructed using Magma\cite{MAGMA} in Section~\ref{sec:prelim}.
        \item The quotient map $\Map(G,\rho,\tau)$ does not necessarily have Euler characteristic equal to a prime power. In Section~\ref{sec:example}, explicit examples of $\Map(X,x,y)$ are provided for all cases covered by Theorem~\ref{thm:main-result-solvable} and Theorem~\ref{thm:main-non-solvable}.
    \end{enumerate}
\end{remark}

Our characterization of the group structure enables the classification of bi-rotary maps of Euler characteristic $-p^n$ for small integer $n$. The classification results for $n\leq 4$ will be presented in a future paper.

\section{Preliminaries}\label{sec:prelim}

In this section, we present fundamental algebraic theory related to bi-rotary maps.
For further reading on algebraic map theory, we refer to \cite{jones2019Automorphism,lando2004Graphs,siran2013How}.

\subsection{Bi-rotary maps: Algebraic theory}\label{subsect:map}
The following lemma gives fundamental properties of a bi-rotary map.
For proof details, see \cite{dazevedoBirotaryMapsNegative2019}.

\begin{lemma}
    \label{lem:fundamental}
    Let $\calm=\Map(X,x, y)$ be a bi-rotary map.
    Then the following statements hold.
    \begin{itemize}
        \item [\rm (i)] $\calm$ is non-orientable if and only if $X=\l [x, y], x\r$. 
        \item [\rm (ii)] The Euler characteristic $\chi$ of $\calm$ is 
 \[\chi=|X|\left(\frac{1}{|x|}-\frac{1}{|y|}+\frac{1}{|\l y, y^x\r|}\right)=|X|\left(\frac{1}{k}-\frac{1}{2}+\frac{1}{m}\right),\]\
 where $(k,m)=(|x|, |\l y, y^x\r|)$ is the type of $\calm$.
        \item [\rm (iii)] Let $\calm_1=\Map(X_1,x_1,y_1)$ be a bi-rotary map. Then $\calm\cong\calm_1$ if and only if there exists an isomorphism $\sigma:X\rightarrow X_1$ such that $\sigma(x)=x_1$ and $\sigma(y)=y_1$.
    \end{itemize}
\end{lemma}

Let $\calm=\Map(X, x, y)$ be a bi-rotary map and suppose $N\lhd X$. The quotient map $\calm/N$ is defined as the map $\Map(X/N, x N, y N)$. Considering the quotient map of a bi-rotary map provides an effective approach to reduce problems to simpler cases via group-theoretic methods. 
In this paper,
as mentioned in the Introduction,
we do not attempt to determine the structure of $\O_p(X)$ in the automorphism group $X$.  
The following example \ref{exm:complexOp(X)} illustrates the difficulty in achieving a complete classification of $\O_p(X)$.
Hence, we focus on determining the structure of the quotient group $X/\O_p(X)$ and the corresponding quotient map $\calm/ \O_p(X)$. Two special degenerate cases arise: (i) $X/\O_p(X)=1$; (ii) $X/\O_p(X)$ is a non-trivial abelian group. We discuss these two special cases in the following subsections.

\begin{example}
    \label{exm:complexOp(X)} 

Let $D=\l x,y| x^3, y^2,[x,y]^4\r$ be a finitely presented group. 
This gives an infinite bi-rotary map $\mathcal{U}=\Map(D, x,y)$. 
As shown in \cite{dazevedoBirotaryMapsNegative2019}, there is a normal subgroup $N$ of $D$ such that $D/N=T\cong \PSL(2,7)$ and this gives the quotient map $\mathcal{U}/N$, which is a bi-rotary map of Euler characteristic $-7$.
Let $N=N_0 \ge N_1\ge N_2\ge \dots$ be the lower exponent-$7$  central series of $N$. By the aids of Magma\cite{MAGMA}, $|N/N_5|=7^{9224}$. For any normal subgroup $L$ of $D$ with $N_5\le L<N$, $\mathcal{U}/L$ is a smooth cover of $\mathcal{U}/N$ whose Euler characteristic is also a power of $7$. Note that the maximal dimension of an irreducible $\mathbb{F}_7T$-module is $7$. Hence, there are at least $\lceil 9224/7 \rceil =1318$ normal subgroups between $N_5$ and $N$. This shows the complexity of the normal subgroup ${\rm O}_p(X)$. 
\end{example}

\subsection{\texorpdfstring{Bi-rotary maps with $p$-groups}{Bi-rotary maps with p-groups}}
If $\calm=\Map(X, x, y)$ is a bi-rotary map and $X/\O_p(X)=1$, then $p=2$ as $y\in \O_p(X)$ is an involution. Hence $X$ is a $2$-group. Now, suppose that $\calm$ is a bi-rotary map of Euler characteristic that is a negative power of two.
Direct computation shows that the type of $\calm$ must be in  $\{ (4,8),(8,4),(8,8) \}$.
In the following, we construct infinitely many bi-rotary maps whose automorphism groups are $2$-groups and Euler characteristic is a power of $2$ for each type.
First, we provide small examples in Example~\ref{example 16} below using groups of order $16$.

\begin{example}\label{example 16}
     Let $X=\l x,y| x^8, y^2, x^yx^5\r=\l x\r{:}\l y\r$ be the semi-dihedral group of order $16$. Then the bi-rotary map $\Map(X,x,y)$ is of type $(8,8)$ and Euler characteristic $\chi=-4$. The bi-rotary map $\Map(X,xy,y)$ is of type $(4,8)$ and Euler characteristic $\chi=-2$. 
    Let $X=\l x,y| x^8, y^2, x^yx^3\r=\l x\r{:}\l y\r$ be a $2$-group of order $16$. Then the bi-rotary map $\Map(X,x,y)$ is of type $(8,4)$ and Euler characteristic $\chi=-2$. 
\qed \end{example}

In the following example, we construct infinitely many bi-rotary maps of Euler characteristic 0, which is a key step in generating such maps for negative Euler characteristic. These examples are adapted from
\cite{bredadazevedoRegularPseudoorientedMaps2015}. 

\begin{example}\label{example 0}
 For any positive integer $f$ and $\varepsilon\in \{0,1\}$, let $U=\Z_{2^{f+\varepsilon}}{\times}\Z_{2^f}=\l u\r{\times}\l v\r$. Set $X=(U{:}\Z_4){:}\Z_2=(U{:}\l x\r){:}\l y\r$ where
 \[u^x=uv^{-1},\ v^x=u^2v^{-1},\ u^y=uv^{-1},\ v^y=v^{-1},\ x^y=u^{-1}vx^{-1}=x^{-1}u^{-1}.\]
 It follows that $X=\l x ,y\r$,  $|X|=2^{2f+3+\varepsilon}$, $|x|=4$, $|y|=2$ and $|[x,y]|=|x^2u^{-1}|=2$. Hence, the bi-rotary map $\Map(X,x,y)$ is of Euler characteristic $0$. This gives infinitely many bi-rotary maps on the torus. 
\qed \end{example}

Now, we are ready to give infinitely many bi-rotary maps of Euler Characteristic $\chi=-2^{f'}$. 

\begin{example}\label{example 2group}
Let $\calm_1=\Map(X_1,x_1,y_1)$ be a bi-rotary map given in Example~\ref{example 16}. That means $|X_1|=16$, and $\calm_1$ is of type $(k,m)\in\{ (4,8),(8,4),(8,8)\}$.
 For any integer $f\ge 5$, let  $\calm_2=\Map(X_2,x_2,y_2)$ be a bi-rotary map of type $(4,4)$ with $|X_2|=2^f$ defined in Example~\ref{example 0}.
 Let $X_0=X_1\times X_2$ and let $x=(x_1,x_2)$, $y=(y_1,y_2)\in X_0$. 
 Now set $X=\l x,y\r$ and let $\calm=\Map(X, x,y)$. It is easy to see that $|x|=\lcm(|x_1,y_2|)=k$ and $|[x,y]|=\lcm(|[x_1,y_1]|,|[x_2,y_2]|)=m/2$. Moreover, as $X\leq X_1{\times}X_2$ has a homomorphic image $X_2$, $2^f\leq|X|\leq 2^{f+4}$. Hence, the bi-rotary map $\calm$ is of type $(k,m)$ and of Euler characteristic $\chi=-2^{f'}$ for some integer $f-3\leq f'\leq f+2$.  This gives infinitely many bi-rotary maps of type $(k,m)$ whose automorphism group is a $2$-group. 
\qed \end{example}

\subsection{Bi-rotary maps with abelian automorphisms}\label{sec:abelian}

In this subsection, we discuss the quotient maps where $X/\O_p(X)$ is a non-trivial abelian group, or equivalently,  bi-rotary maps with abelian automorphism groups. 
  We first give two families of graphs.
  A graph with only one vertex and $n$ loops is called a \emph{bouquet} with $n$ edges and denoted by ${\sf B}_n$.
  A graph with two vertices and $n$ multiple edges joining them is called a \emph{dipole graph} with $n$ edges and denoted by ${\sf D}_n$. The following two families of maps are bi-rotary maps with abelian automorphism groups whose underlying graphs are ${\sf B}_n$ and ${\sf D}_n$, respectively.

  \begin{example}
   Let $n$ be a positive integer and let $X=\l x\r\cong \Z_{2n}$. Set $y=x^n$. Then 
    $\Map(X,x,y)$ has one vertex, $n$ edges and $n$ faces. Thus, it
     is an embedding of ${\sf B}_n$ on the projective plane. Denote this map by ${\mathcal B}_n$. See Figure \ref{fig: dipole on sphere} for an example with $n=3$. 
  \qed \end{example}
  \begin{example}
   Let $n$ be a positive integer and let $X=\l x\r{\times} \l y\r\cong \Z_n{\times}\Z_2$. Then 
    $\Map(X,x,y)$ has two vertices, $n$ edges and $n$ face. Thus, it
     is an embedding of ${\sf D}_n$ on the sphere. Denote this map by ${\mathcal D}_n$. See Figure \ref{fig: dipole on sphere} for an example with $n=6$. 
  \qed \end{example}
 
 The following easy proposition shows that any bi-rotary map with an abelian automorphism group belongs to one of the previous two families. 

 \begin{proposition}
  Let $\calm$ be a bi-rotary map with $n$ edges and an abelian automorphism group. Then either $\calm\cong {\mathcal B}_n$ or $\calm \cong {\mathcal D}_n$.
 \end{proposition}
  \begin{proof}
      Suppose that $\calm=\Map(X,x,y)$ where $X$ is abelian. Note that $X=\l x,y\r$. Hence, either $X=\l x\r$ or $X=\l x\r{\times} \l y\r$. If the former case holds, then clearly $\calm\cong {\mathcal B}_n$. If the latter case holds, then $\calm \cong {\mathcal D}_n$.
     \end{proof}   

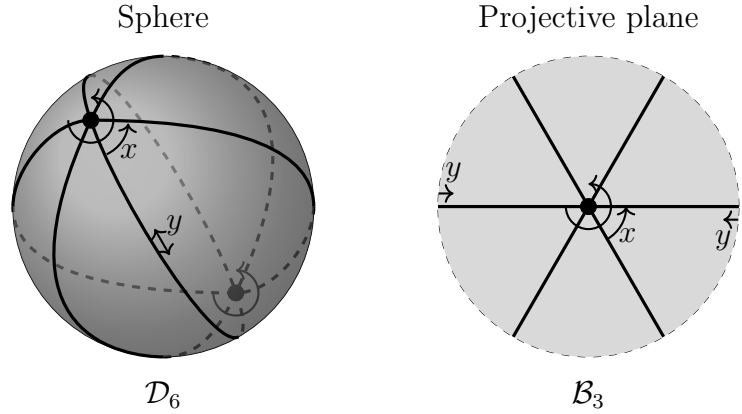
\begin{figure}[ht]
    \centering
     \begin{tikzpicture} 
      \draw[thin] (0,0) circle (2cm);
      \draw[very thick,dashed] (120:2cm) ..controls ($(120:2cm)-(210:0.5)$) and ($(-60:2cm)-(210:.5)$) ..(-60:2cm);
      \filldraw[fill=black] (-50:1.5) circle (3pt);
      \draw[thick,->] ($(-50:1.5)+(-0.3,0)$) arc (180:450:0.3cm);
      \draw[very thick,dashed]  (-90:2cm) ..controls +(180:-0.5) and +(65:-0.5) ..(-50:1.5cm) .. controls +(65:3.2) and +(180:-0.5) .. (90:2);
      \draw[very thick,dashed] (0:2cm) ..controls +(90:-0.5) and +(180:-0.5) ..(-50:1.5cm) .. controls +(0:-2.5) and +(-90:0.5) .. (180:2);
      \fill[ball color=gray!10,opacity=0.5] (0,0) circle (2cm);
      \draw[very thick] (120:2cm) ..controls +(210:0.5) and +(210:.5) ..(-60:2cm); 
      \draw[thick,<->] ($(120:-0.2)+(210:.25)$)--($(-60:.6)+(210:.25)$);
      \node[anchor= south west] at ($(120:-0.4)+(210:.35)$) {$y$};
      \filldraw[fill=black] (130:1.5) circle (3pt);
      \draw[ thick,->] ($(130:1.5)+(-0.3,0)$) arc (180:450:0.3cm);
      \draw[ thick,->] ($(130:1.5)+(-65:.5)$) arc (-65:-5:0.5cm);
      \node[anchor= north west] at ($(130:1.5)+(-35:.3)$) {$x$};
      \draw[very thick] (90:2cm) ..controls +(180:0.5) and +(65:0.5) ..(130:1.5cm) .. controls +(65:-3.2) and +(180:0.5) .. (-90:2);
      \draw[very thick] (180:2cm) ..controls +(90:0.5) and +(180:0.5) ..(130:1.5cm) .. controls +(180:-2.5) and +(90:0.5) .. (0:2);
      \node at (0,-2.5cm) { ${\mathcal D}_6$};
      \node at (0,2.5cm) { Sphere};
    \end{tikzpicture}
\hspace*{1cm}
         \begin{tikzpicture} 
      \draw[dashed] (0,0) circle (2cm);
      \fill[color=gray!30] (0,0) circle (2cm);
      \foreach \i in {0,1,2} {
      \draw[very thick] ($(0+\i*60:2)$)--($(180+\i*60:2)$); 
      };
      \filldraw[fill=black] (0,0) circle (3pt);
      \draw[thick,->] (180:.3cm) arc (180:450:0.3cm);
      \draw[thick,->] ($(0,0)+(-60:.5)$) arc (-60:0:0.5cm);
      \node[anchor= north west] at (-30:.3) {$x$};
      \draw[thick,->] (175:2)--($(175:2)+(0.2,0)$);
      \draw[thick,->] (-5:2)--($(-5:2)-(0.2,0)$);
      \node[anchor= south] at ($(175:2)+(.2,0)$) {$y$};
      \node[anchor= north] at ($(-5:2)-(.2,0)$) {$y$};
      \node at (0,-2.5cm) { ${\mathcal B}_3$};
      \node at (0,2.5cm) { Projective plane};
    \end{tikzpicture}
    \caption{Two bi-rotary maps with abelian automorphism groups}
    \label{fig: dipole on sphere}
\end{figure}

Although these maps have a positive Euler characteristic, they naturally emerge during the quotient map process. 
There is also a degenerate case that arises during quotient processes, specifically when \( y = 1 \). Note that in this case $X=\l x,y\r=\l x\r\cong \Z_n$ can be viewed as a quotient of the group $\Z_n\times \Z_2$. Hence, the degenerate map $\Map(X,x,y)$ can thus 
be viewed as gluing the northern and southern hemispheres of $\mathcal{D}_{n}$
along the equatorial plane to form a single disk, as illustrated in Figure~\ref{fig: semistar}. This is an embedding of a semistar with $n$ semi-edges on a closed disk.

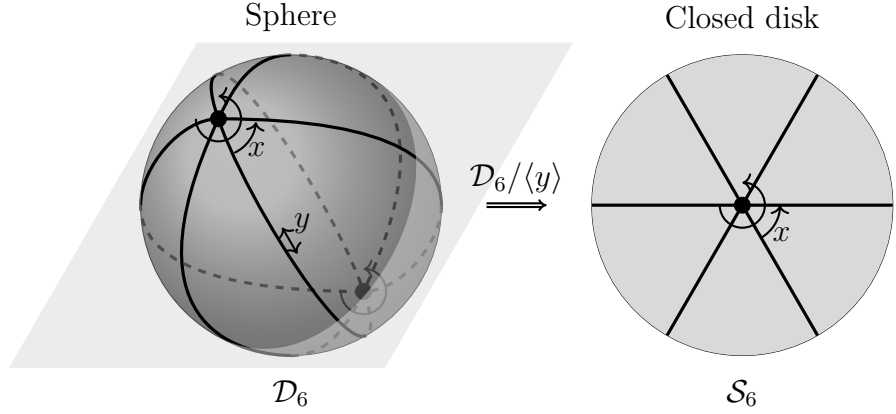
\begin{figure}[ht]
    \centering
     \begin{tikzpicture} 
      \draw[thin] (0,0) circle (2cm);
      \draw[very thick,dashed] (120:2cm) ..controls ($(120:2cm)-(210:0.5)$) and ($(-60:2cm)-(210:.5)$) ..(-60:2cm);
      \filldraw[fill=black] (-50:1.5) circle (3pt);
      \draw[thick,->] ($(-50:1.5)+(-0.3,0)$) arc (180:450:0.3cm);
      \draw[very thick,dashed]  (-90:2cm) ..controls +(180:-0.5) and +(65:-0.5) ..(-50:1.5cm) .. controls +(65:3.2) and +(180:-0.5) .. (90:2);
      \draw[very thick,dashed] (0:2cm) ..controls +(90:-0.5) and +(180:-0.5) ..(-50:1.5cm) .. controls +(0:-2.5) and +(-90:0.5) .. (180:2);
      \fill[ball color=gray!10,opacity=0.5] (0,0) circle (2cm);
      \draw[very thick] (120:2cm) ..controls +(210:0.5) and +(210:.5) ..(-60:2cm); 
      \draw[thick,<->] ($(120:-0.2)+(210:.25)$)--($(-60:.6)+(210:.25)$);
      \node[anchor= south west] at ($(120:-0.4)+(210:.35)$) {$y$};
      \filldraw[fill=black] (130:1.5) circle (3pt);
      \draw[ thick,->] ($(130:1.5)+(-0.3,0)$) arc (180:450:0.3cm);
      \draw[ thick,->] ($(130:1.5)+(-65:.5)$) arc (-65:-5:0.5cm);
      \node[anchor= north west] at ($(130:1.5)+(-35:.3)$) {$x$};
      \draw[very thick] (90:2cm) ..controls +(180:0.5) and +(65:0.5) ..(130:1.5cm) .. controls +(65:-3.2) and +(180:0.5) .. (-90:2);
      \draw[very thick] (180:2cm) ..controls +(90:0.5) and +(180:0.5) ..(130:1.5cm) .. controls +(180:-2.5) and +(90:0.5) .. (0:2);
      \fill[color=gray!30,opacity=0.5] (60:-2).. controls +(-30:2) and +(-30:2) ..(60: 2) arc (60:180:2)-- (-2.5,0)--++(60:2.5)--++(5,0)--++(-120:5)--++(-5,0)--(-2.5,0)--(-2,0) arc (-180:-120:2) --cycle;
      \node at (0,-2.5cm) { ${\mathcal D}_6$};
      \node at (0,2.5cm) { Sphere};

      \draw (6,0) circle (2cm);
      \fill[color=gray!30] (6,0) circle (2cm);
      \foreach \i in {0,1,2} {
      \draw[very thick] ($(\i*60:2)+(6,0)$)--($(180+\i*60:2)+(6,0)$); 
      };
      \filldraw[fill=black] (6,0) circle (3pt);
      \draw[thick,->] ($(180:.3cm)+(6,0)$) arc (180:450:0.3cm);
      \draw[thick,->] ($(6,0)+(-60:.5)$) arc (-60:0:0.5cm);
      \node[anchor= north west] at ($(-30:.3)+(6,0)$) {$x$};
      \draw[double,thick,-{Implies[]}] (2.6,0)--(3.4,0);
      \node[anchor=south] at (3,0) {${\mathcal D}_6/\l y\r$};
      \node at (6,-2.5cm) { ${\mathcal S}_6$};
      \node at (6,2.5cm) { Closed disk};
    \end{tikzpicture}

    \caption{Getting a degenerate map ${\mathcal S}_6$ by taking quotient of ${\mathcal D}_6$}
    \label{fig: semistar}
\end{figure}


\section{Basic properties of bi-rotary maps with negative prime power Euler characteristic}\label{sec:property}

As seen in the previous section, the automorphism group of the bi-rotary map determines its fundamental properties. Therefore, bi-rotary maps with Euler characteristic $-p^n$ can be entirely studied within the framework of group theory. To apply techniques from quotient and subgroup structure analysis in group theory, we define the following related group classes.

\begin{definition}
 Let $G$ be a finite group and let $p$ be a prime.
 \begin{enumerate}
 \item The group $G$ is called a $\mathscr{P}_0(p)$-group with respect to a pair $(\rho,\tau)\in G{\times}G$ if  $|\tau|=2$, $\l \rho,\tau\r=G$, and
 \[|G|\left(\frac{1}{|\rho|}+\frac{1}{|\l \tau, \tau^\rho\r|}-\frac{1}{|\tau|}\right)=-p^n, \]
 for some positive integer $n$.
 \item The group $G$ is called a  $\mathscr{P}_1(p)$-group if there exist two subgroups $H_1$ and $H_2$ of $G$ which are cyclic or dihedral such that 
 \[|G|=p^n{\cdot}\lcm(|H_1|, |H_2|),\]
 for some non-negative integer $n$. To be more specific, $G$ is called a $\mathscr{P}_1(p)$-group with respect to $H_1$ and $H_2$.
 \item The group $G$ is called a $\mathscr{P}_1^+(p)$-group with respect to a pair $(\rho,\tau)\in G{\times}G$ if $|\tau|\le 2$, $\l \rho,\tau\r=G$, $\l [\rho,\tau],\rho\r=G$ if $p$ is odd, and $G$ is a $\mathscr{P}_1(p)$-group with respect to $\l \tau,\tau^\rho\r$ and $\l\rho\r$.
 \item The group $G$ is called a $\mathscr{P}_2(p)$-group if for each prime $r\neq p$, the Sylow $r$-subgroup of $G$ is cyclic or dihedral. 
 \item The group $G$ is called a $\mathscr{P}^+_2(p)$-group with respect to a pair $(\rho,\tau)\in G{\times}G$ if $|\tau|\le 2$, $\l \rho,\tau\r=G$, $\l [\rho,\tau],\rho\r=G$ if $p$ is odd, and $G$ is a $\mathscr{P}_2(p)$-group.
 \end{enumerate}
\end{definition}

The following Lemma~is a group theory version of Lemma~\ref{lem:fundamental} {\rm (i)} which is useful in analyzing group structure. 

\begin{lemma}\label{lem:p0orient}
 If $G$ is $\mathscr{P}_0(p)$-group with respect to a pair $(\rho,\tau)$, where $p$ is an odd prime, then $\l [\rho,\tau], \rho\r=G$. 
\end{lemma}

\begin{proof}
 It follows immediately from the definition that the bi-rotary map $\Map(G, \rho,\tau)$ is of Euler characteristic $-p^n$ for some positive integer $n$. As $p$ is odd, we have $-p^n$ is odd. By the Euler-Poincar\'e formula, we have $\Map(G, \rho,\tau)$ is non-orientable and hence by Lemma~\ref{lem:fundamental} {\rm (i)}, $\l [\rho,\tau], \rho\r=G$. 
\end{proof}

By Lemma~\ref{lem:p0orient}, we have the following result, which reveals the relationship between $\mathscr{P}^+_i(p)$ for $i\in\{1,2\}$ and $\mathscr{P}_j(p)$-groups for $j\in \{0,1,2\}$.

\begin{lemma}\label{relationpi}
 Let $G$ be a finite group and let $p$ be a prime. 
 \begin{itemize}
  \item [\rm (i)] If $G$ is a $\mathscr{P}_0(p)$-group with respect to $(\rho,\tau)$, then it is a $\mathscr{P}_1^+(p)$-group with respect to $(\rho,\tau)$.
  \item [\rm (ii)] If $G$ is a $\mathscr{P}_1(p)$-group, then it is a $\mathscr{P}_2(p)$-group.
  \item [\rm (iii)] If $G$ is a $\mathscr{P}_1^+(p)$-group with respect to $(\rho,\tau)$, then it is also a $\mathscr{P}^+_2(p)$-group with respect to $(\rho,\tau)$.
 \end{itemize}
\end{lemma}

\begin{proof}
  Suppose $G$ is a $\mathscr{P}_0(p)$-group with respect to $(\rho,\tau)$.
That is $|\tau|=2$, $\l\rho,\tau\r=G$, and
\[|G|\left(\frac{1}{|\rho|}+\frac{1}{|\l \tau, \tau^\rho\r|}-\frac{1}{|\tau|}\right)=\left(\frac{|G|}{|\rho|}+\frac{|G|}{|\l \tau, \tau^\rho\r|}-\frac{|G|}{|\tau|}\right)=-p^n, \]
for some positive integer $n$. 
By Lemma~\ref{lem:p0orient}, $\l[\rho,\tau],\rho\r=G$ if $p$ is odd.
Set $H_1=\l\rho\r$ and $H_2=\l \tau, \tau^\rho\r$. 
It is sufficient to show that  \[|G|=p^{n'}{\cdot}\lcm(|H_1|, |H_2|),\]
 for some non-negative integer $n'$. That is, we need to show $p$ is the only prime divisor of $|G|/\lcm(|H_1|, |H_2|)$.
   Let $r$ be a prime divisor of 
   \[\frac{|G|}{\lcm(|H_1|, |H_2|)}=\gcd\left(\frac{|G|}{|H_1|}, \frac{|G|}{|H_2|}\right).\]
   Since $2=|\tau|$ is a divisor of  $|H_2|$, we have 
\[\gcd\left(\frac{|G|}{|H_1|},\frac{|G|}{|H_2|}\right)=\gcd\left(\frac{|G|}{|H_1|},\frac{|G|}{|H_2|},\frac{|G|}{2}\right).\]
   Hence, $r$ must be a prime divisor of $-p^n$, which implies that $r=p$, and $G$ is a $\mathscr{P}_1(p)$-group with respect to $H_1$ and $H_2$.

  Suppose $G$ is a $\mathscr{P}_1(p)$-group. There are two subgroups $H_1, H_2$ of  $G$ which are cyclic or dihedral such that 
\[\frac{|G|}{\lcm(|H_1|,|H_2|)}=\gcd\left(\frac{|G|}{|H_1|}, \frac{|G|}{|H_2|}\right)=p^n,\]
  for some non-negative integer $n$. 
   Let $r$ be a prime divisor of $|G|$ other than $p$. Then either $|G|/|H_1|$ or $|G|/|H_2|$ is not divisible by $r$. Hence either $H_1$ or $H_2$ contains a Sylow $r$-subgroup of $G$. As $H_1$ and $H_2$ are cyclic or dihedral, we have that all Sylow $r$-subgroups of $G$ are cyclic or dihedral, which means $G$ is a $\P{2}{p}$-group.

   Suppose $G$ is a $\mathscr{P}_1^+(p)$-group with respect to $(\rho,\tau)$. By {\rm (ii)} and the definition of $\mathscr{P}_2^+(p)$-groups, $G$ is a $\mathscr{P}_2^+(p)$-group with respect to $(\rho,\tau)$.
\end{proof}

Let $\calm=\Map(G,\rho,\tau)$ be a bi-rotary map. 
As shown in Subsection~\ref{subsect:map}, $\calm$ has an Euler characteristic  $-p^n$ for some prime $p$ if and only if $G$ is a $\mathscr{P}_0(p)$-group with respect to the pair $(\rho,\tau)$. 
Hence, studying automorphism groups of bi-rotary maps with negative prime power Euler characteristic is equivalent to analyzing $\mathscr{P}_0(p)$-groups for some prime $p$.
It is easy to see that a subgroup or a quotient group of a $\P{0}{p}$-group is not necessarily a $\P{0}{p}$-group, but the following lemmas show that being an $\P{1}{p}$-group and being an $\P{2}{p}$-group are inherited in some sense.

\begin{lemma}\label{lem:p1}
Let $G$ be a $\mathscr{P}_1(p)$-group with respect to $H_1$ and $H_2$. 
If $N$ is a normal subgroup of $G$, then
\begin{itemize}
    \item [\rm (i)] $N$ is a $\mathscr{P}_1(p)$-group with respect to $H_1\cap N$ and $H_2\cap N$;
    \item [\rm (ii)] $G/N$ is a $\mathscr{P}_1(p)$-group with respect to $H_1N/N$ and $H_2N/N$.
\end{itemize}
In particular, the class of $\mathscr{P}_1(p)$-groups is closed under normal subgroups and quotients.
\end{lemma} 

\begin{proof}
 By definition, we have
\[\frac{|G|}{\lcm(|H_1|,|H_2|)}=\gcd\left(\frac{|G|}{|H_1|}, \frac{|G|}{|H_2|}\right)=p^n,\] for some non-negative integer $n$. 
Set $\overline{G}=G/N$, $\overline{H}_i=H_iN/N$, and $L_i=H_i\cap N$ for $i\in \{1,2\}$. 
As $N$ is a normal subgroup of $G$, $H_iN$ is a subgroup of $G$ and both $|H_iN{:} H_i|$, $|G{:}H_iN|$ are divisors of $|G{:}H_i|$, where $i\in \{1,2\}$. It follows that  both 
\[\gcd\left(\frac{|H_1N|}{|H_1|},\frac{|H_2N|}{|H_2|}\right)=\gcd\left(\frac{|N|}{|L_1|},\frac{|N|}{|L_2|}\right)\]
and \[  \gcd\left(\frac{|G|}{|H_1N|},\frac{|G|}{|H_2N|}\right)=\gcd\left(\frac{|\overline{G}|}{|\overline{H}_1|},\frac{|\overline{G}|}{|\overline{H}_2|}\right)\]
are divisor of $p^n$, which is also a power of $p$. It is obvious that $\overline{H}_1,\overline{H}_2,L_1,L_2$ are either cyclic or dihedral. This completes the proof. 
\end{proof}

\begin{corollary}\label{cor:p1+}
    Let $G$ be a $\mathscr{P}_1^+(p)$-group with respect to $(\rho,\tau)$ and let $N$ be a normal subgroup of $G$.
    Then $G/N$ is a $\mathscr{P}_1^+(p)$-group with respect to $(\rho N,\tau N)$.
\end{corollary}

\begin{proof}
    By Lemma~\ref{lem:p1} {\rm (ii)},
    $G/N$ is a $\mathscr{P}_1(p)$-group with respect to $\l \tau N,\tau^\rho N\r$ and $\l \rho N\r$.
    Then by the definition of $\mathscr{P}^+_1(p)$-groups,
    $G/N$ is a $\mathscr{P}_1^+(p)$-group with respect to $(\rho N,\tau N)$.
\end{proof}

  Recall that a section of a group $G$ is a quotient group of a subgroup of $G$. 

\begin{lemma}\label{lem:p2}
 Let $G$ be a $\mathscr{P}_2(p)$-group. Then any section $H$ of $G$ is also a $\mathscr{P}_2(p)$-group.
\end{lemma}

\begin{proof}
 We only need to show that $H$ and $G/N$ are $\mathscr{P}_2(p)$-groups for any subgroup $H$ and any normal subgroup $N$ of $G$.
 For any prime $r$ other than $p$, let $S_1, S_2$ be a Sylow $r$-subgroups of $G$ and $H$, respectively. 
 Then, there is an element $g\in G$ such that $S_2\le S_1^g$. 
 Hence, $S_2$ is also cyclic or dihedral.  
 Note that $S_1N/N$ is a Sylow $r$-subgroup of $G/N$. Thus, the Sylow $r$-subgroups of $G/N$ are also cyclic or dihedral.
 Since $r$ is an arbitrary prime, we have that $H$ and $G/N$ are also a $\mathscr{P}_2(p)$-groups. 
\end{proof}

\begin{corollary}
    \label{cor:p2+}
    Let $G$ be a $\mathscr{P}_2^+(p)$-group with respect to $(\rho,\tau)$
    and let $N$ be a normal subgroup of $G$.
    Then $G/N$ is a $\mathscr{P}_2^+(p)$-group with respect to $(\rho N,\tau N)$.
\end{corollary}
\begin{proof}
    By Lemma~\ref{lem:p2}, $G/N$ is a $\mathscr{P}_2(p)$-group.
    Then by the definition of $\mathscr{P}_2^+(p)$-groups,
    $G/N$ is a $\mathscr{P}_2^+(p)$-group with respect to $(\rho N,\tau N)$.
\end{proof}

\section{Solvable automorphism groups}\label{sec:solvable}
In this section, we aim to prove Theorem~\ref{thm:main-result-solvable}. 
Let $\Map(X,x,y)$ be a bi-rotary map of Euler characteristic $-p^n$, let $G=X/\O_p(X)$,  let $\rho=x\O_p(X)$ and let $\tau=y\O_p(X)$. 
Then $G$ is a solvable $\mathscr{P}^+_1(p)$-group with respect to $(\rho,\tau)$ and $\O_p(G)=1$. 
The case where $G$ is abelian has been discussed in Section \ref{sec:abelian}.
Therefore, we suppose that $G$ is non-abelian.
First, we state several useful lemmas on $\mathscr{P}^+_2(p)$-groups.

\begin{lemma}\label{solvable-gp-structure}
    Let $G$ be a non-abelian solvable $\mathscr{P}^+_{2}(p)$-group with respect to $(\rho,\tau)$ and suppose $\O_p(G)=1$. 
    Let $K=G_{\{2,3\}'}$ and $H=G_{\{2,3\}}$ be a Hall $\{2,3\}'$-subgroup and a Hall $\{2,3\}$-subgroup of $G$, respectively.
    Then the following holds.
    \begin{itemize}
        \item [\rm (i)] $G_p$ is isomorphic to $1$, or  $\Z_{p^e}$, or $\Z_2\times \Z_{2^e}$, or $\Z_3\times \Z_{3^e}$ for some positive integer $e$.
        \item [\rm (ii)] $K\lhd G$ and  
        \[G=K{:}H=\l a \r{:}(\l b\r \times H),\]
        where $\l a \r =[K,H]=[\l a\r, H]$, $\l b\r =C_K(H)$ and $\gcd(|a|,|b|)=1$.
    \end{itemize}
\end{lemma}

\begin{proof}
    Let $F$ be the Fitting subgroup of $G$. 
    Then $F=F_2\times F_{2'}$ and 
    $G/F\lesssim\Out(F_2)\times\Out(F_{2'})$. 
    As $G$ is solvable, $C_G(F)\leq F$.   
    Set $\overline{G}=G/F$. 
    Then  $\overline{G}$ is also generated by two elements $\overline{\rho}, \overline{\tau}$ with $|\overline{\tau}|\le 2$. 
    Since $G$ is a $\P{2}{p}$-group and $\O_p(G)=1$, for any prime $r>2$, $\O_r(G)$ is cyclic, and so thus $F_{2'}$ is also cyclic.  
    Therefore, $\Out(F_{2'})$ is abelian.
    Since $F_2$ is cyclic or dihedral, 
    $\Out(F_2)$ is abelian  except for the case $F_2=\Z_2^2$ and $\Out(F_2)\cong \S_3$. It follows immediately that all Sylow subgroups of $\overline{G}$ are abelian. Let $G_p$ be a Sylow $p$-subgroup of $G$. We have $G_pF/F$ is a Sylow $p$-subgroup of $\overline{G}$, and 
    $G_pF/F\cong G_p/(G_p\cap F)=G_p/\O_p(G)\cong G_p$ as $\O_p(G)=1$. 
    If $\overline{G}$ is abelian, then $G_p$ is isomorphic to $1$ or $\Z_{p^e}$ or $\Z_2\times\Z_{2^e}$ for some positive integer $e$ as $\overline{G}=\l \overline{\rho}, \overline{\tau}\r$ and $|\overline{\tau}|\le 2$.
    If $\overline{G}$ is non-abelian, then $1\ne \overline{G}'\le (S_3{\times} \Out(F_{2'}))'\cong\Z_3$. This gives $\overline{G}'\cong \Z_3$. If $p\ne 3$, then $G_p$ is isomorphic to a Sylow $p$-subgroup of $\overline{G}/\overline{G}'$ and again  $G_p$ is isomorphic to $1$ or $\Z_{p^e}$ or $\Z_2\times\Z_{2^e}$ for some positive integer $e$. If $p=3$, by similar argument, we have the Sylow $3$-subgroup of  $\overline{G}/\overline{G}'$ is cyclic. Therefore, $G_p=G_3\cong \overline{G}_3=\overline{G}'.\Z_{3^e}=\Z_3.\Z_{3^e}$. 
    Thus $G_3\cong \Z_{3^{e+1}}\mbox{ or }\Z_{3}\times\Z_{3^e}$. This completes the proof of part (i).

    Now, by part (i), all Sylow subgroups of $G$ are metacyclic. By \cite[Theorem~1]{Sylow-metacyclic1981},  we have $K\lhd G$.   
    Consider the conjugacy action of $H$ on $K$. As pointed out in \cite[p.176]{kurzweil2004theory}, $[K,H]$ is an $H$-invariant  normal subgroup of $K$, so is a normal subgroup of $G$. Obviously, $[K,H]$ is a $\{2,3\}'$-subgroup of $G'$. 
     Recall that either $\overline{G}=G/F$  is abelian or $\overline{G}'\cong \Z_3$.
     Therefore,  $[K,H]\le (G'\cap K)\leq (F\cap K)\leq F_{2'}$. This implies that $[K,H]$ is cyclic.  
     Let $[K,H]=\l a\r$. 
    Since $\gcd(|H|,|K|)=1$, by \cite[p.187 8.2.7]{kurzweil2004theory},  $K=[K,H]C_{K}(H)=\l a\r C_{K}(H)$ and $[K,H,H]=[K,H]$. It follows that $K=\l a \r C_{K}(H)$ and $[\l a\r , H]=[ [K,H],H]=[K,H,H]=[K,H]=\l a\r$. Now consider the coprime action of $H$ on $\l a\r$. By \cite[p.198 8.4.2]{kurzweil2004theory}, 
    \[\l a \r =[\l a \r, H]\times  C_{\l a \r}(H)=\l a \r \times C_{\l a\r}(H).\] This gives $1=C_{\l a\r}(H)=C_{K}(H)\cap \l a\r$ and it means $G=\l a\r {:}(C_{K}(H)\times H)$. Moreover $C_{K}(H)$ is a $\{2,3\}'$-group which is also a quotient of $G$. Hence $C_{K}(H)$ is cyclic. Set $C_{K}(H)=\l b\r$. By part (i), all Sylow subgroup of $K=\l a\r {:}\l b\r$ are cyclic, and therefore, $\gcd(|a|,|b|)=1$. Therefore
    \[G=\l a \r {:}(\l b\r \times H).\]
    This completes the proof of part (ii). 
\end{proof}

Lemma~\ref{solvable-gp-structure} {\rm (ii)} plays a fundamental role in analyzing the structure of the group $G$. 
We determine the possible structure of the Hall $\{2,3\}$-subgroup $H$ of $G$ in the following two lemmas. 

\begin{lemma}\label{hall23-normal}
 Let $G,H,a,b$ be as described in Lemma~\ref{solvable-gp-structure}(ii) and let $H_3$ be a Sylow $3$-subgroup of $H$. Suppose that $H$ is not abelian and  $ H_3\lhd H$. Then $H$ is isomorphic to $\D_{2{\cdot}3^e}$ or $\D_{2{\cdot}3^e}\times\Z_{2^f}$, 
   where $e,f$ are positive integers.
\end{lemma}

\begin{proof}
  Let $H_2$ be a Sylow $2$-subgroup of $H$. 	By Lemma~\ref{solvable-gp-structure}(i), $H_3\cong  \Z_{3^e}$ or $\Z_{3}\times \Z_{3^e}$, and $H_2\cong \Z_{2^f}, \D_{2^{2+f}}$ or $\Z_2\times \Z_{2^f}$, where $e,f$ are positive integers. 

  	 In the case where $H_3\cong \Z_{3}\times \Z_{3^e} $, we have  $p=3$. If $H_2$ acts non-trivially on $H_3$, then $H'\cap H_3\neq 1$ is a normal subgroup in $H$. Note that $\Aut(\l a\r)$ is abelian. The centralizer $C_H(\l a\r )\ge H'\cap H_3$. Thus $1\neq H'\cap H_3$ is centralized by both $a$ and $b$ and normal in $H$, which makes it a normal subgroup of $G$. This contradicts $\O_p(G)=1$. If $H_2$ acts trivially on $H_3$, then $H_3$ is also a homomorphic image of $G$. That means $H_3$ can be generated by two elements, one of which has an order less than or equal to $2$. This contradicts  $H_3\cong \Z_{3}\times \Z_{3^e} $. 

  	 For the case $H_3\cong \Z_{3^e}$, note that $H_2$ is also a homomorphic image of $G$. 
  	 There exist elements $\overline{\rho}, \overline{\tau}\in H_2$ with $|\overline{\tau}|\le 2$ such that $\l \overline{\rho},\overline{\tau}\r=H_2$. 
  	 The automorphism groups $\Aut(H_3)$ and $\Aut(\l a\r)$ are abelian. Hence $H'_2$ centralize both $a, b$ and $H_3$, which implies $H'_2\lhd G$. 
  	 If $H'_2\ne1$, then $\O_2(G)\ge H'_2$ is not trivial. 
  	 But, when $H'_2\ne 1$, we have $H_2\cong \D_{2^{2+f}}$ and $\l \overline{\rho},[\overline{\rho}, \overline{\tau}]\r<H_2$. Then $p=2$ and $\O_2(G)=1$, which is a contradiction. 
  	 Thus, we have $H_2\cong \Z_{2^f}$ or $\Z_2\times \Z_{2^f}$. 	
  	  Recall that $H$ is a homomorphic image of $G$. 
	So, there exist $\rho'\in H$ and $\tau'\in H_2$ with $|\tau'| = 2$ such that $\l \rho',\tau'\r=H$.
  	  Now set  $C=C_{H_2}(H_3)$. 
  	  It is easy to see that $C\lhd H$. If $\tau'\in C$, then we have $H/C$ is cyclic, which implies that $H_2/C$ acts trivially on $H_3$. This gives $C=H_2$ and $H=\Z_{2^f{\cdot}3^e}$ or $\Z_{2^f{\cdot} 3^e}\times \Z_2 $. Now suppose that $\tau'\notin C$. Note that the Sylow $2$-subgroup of  $\Aut(H_3)$ is isomorphic to $\Z_2$. It follows that $H\cong(\Z_{3^e}\times C){:}\l \tau'\r$ and $H_2=C{:}\l\tau'\r=C\times \l \tau'\r$. 
      If $H_2\cong \Z_{2^f}$, then $C=1$ and $f=1$. Therefore  $H\cong \D_{2{\cdot} 3^e}$.  
      If $H_2\cong \Z_{2^f}\times \Z_2$, then  $C=\Z_{2^{f}}$. Therefore $H\cong \D_{2{\cdot}3^e}\times\Z_{2^f}.$
\end{proof}

\begin{lemma}\label{hall23-not-normal}
 Let $G,H,a,b$ be  as described in Lemma~\ref{solvable-gp-structure}(ii) and let $H_3$ be a Sylow $3$-subgroup of $H$. Suppose that $H_3\notlhd H$. Then $H$ is isomorphic to one of the following groups:
   \[\Z_2^2{:}\Z_{3^e},\  \S_4{\times }\Z_{3^{e}},\  \Z_2^2{:}\D_{2{\cdot} 3^e}, \]
   where $e$ is a positive integer.
\end{lemma}

\begin{proof}
  	Let $\overline{H}=H/\O_3(H)$. We first claim that $\overline{H}\cong \A_4$ or $\S_4$. 
   Since $H_3$ is not normal in $H$, $\overline{H}_3\neq 1$ and $\O_3(\overline{H})=1$. Set $\overline{N}=\O_2(\overline{H})$ and
   let $\overline{M}$ be the preimage of $\O_3(\overline{H}/\overline{N})$ in $\overline{H}$. Thus $\overline{M}=\overline{N}{:}\overline{M}_3$. 
   Let $\overline{C}=C_{\overline{M}}(\overline{N})$.
   Since both $\overline{M}$ and $\overline{N}$ are normal in $\overline{H}$, $\overline{C}=C_{\overline{H}}(\overline{N})\cap \overline{M}\lhd \overline{H}$. 
   Note that $\overline{C}\cap \overline{M}_3$
   is a characteristic subgroup of $\overline{C}$.
   We claim that $\overline{C}\cap \overline{M}_3= 1$.
   Otherwise, 
   $\overline{C}\cap \overline{M}_3$ is a non-trivial normal $3$-subgroup of $\overline{H}$, which contradicts with $\O_3(\overline{H})=1$. 
   Therefore, 
   the conjugacy action of $\overline{M}_3$ on $\overline{N}$ is faithful. Note that $\Aut(\overline{N})$ is a $2$-group except for the case $\overline{N}\cong \Z_2^2$. 
   Thence $\overline{M}_3\cong \Z_3$  and $\overline{M}\cong \A_4$.
   Now, consider the conjugacy action of $\overline{H}$ on the four $\Z_3$ in $\overline{M}$ with kernel $\overline{L}$. 
   Note that $\overline{M}$ acts faithfully on these four groups. Hence $\overline{L}\cap \overline{M}=1$. But $\O_3(\overline{H})=1$ and $\O_2(\overline{H})=\overline{N}\le \overline{M}$ which gives $\O_3(\overline{L})=\O_2(\overline{L})=1$. It follows that $\overline{L}=1$ and $\overline{H}\cong \A_4$ or $\S_4$. 

   Now, note that every Sylow $3$-subgroup of $H$ is abelian and contains $\O_3(H)$. Hence, $C_H(\O_3(H))$ contains the normal closure of $H_3$.
   Also, notice that the normal closure of $H_3$ is the preimage of $\overline{M}$. Hence, the preimage of $\overline{N}$ is $\O_2(H)\times \O_3(H)$. If $\overline{H}\cong \A_4$, then $H_3\cong H/\O_2(H)$ is a homomorphic image of $G$. This implies $H_3$ is cyclic and $H=\Z_2^2{:}\Z_{3^e}$.
 If $\overline{H}\cong \S_4$ and $H_3\cong \Z_{3^e}$, then $H\cong  \Z_2^2{:}\D_{2{\cdot} 3^e}$
    If $\overline{H}\cong \S_4$ and $H_3\cong \Z_3\times \Z_{3^e}$, then $p=3$. Note that $\Aut(\l a\r)$ is abelian. The commutator subgroup $H'$ centralizes both $a$ and $b$. Thus $H'\cap \O_3(H)$ is a normal subgroup of $G$. Therefore, $H'\cap \O_3(H)\le \O_p(G)=1$. Thus 
   \[\A_4=\S_4'\cong \overline{H}'=H'\O_3(H)/\O_3(H)\cong H'/(\O_3(H)\cap H')\cong H',\] and \[H'\O_3(H)=H'\times \O_3(H)=\A_4 \times \Z_{3^e}.\]
   Moreover, $H=\S_4\times \Z_{3^e}$. 
\end{proof}

\begin{lemma}\label{lem:P1psolvable}
    Let $G$ be a non-abelian solvable $\mathscr{P}^+_{1}(p)$-group with respect to $(\rho,\tau)$ and suppose $\O_p(G)=1$. 
    Let $H$ be a Hall $\{2,3\}$-subgroup of $G$.
    Then $H$ is isomorphic to one of the following groups
    \[\Z_{k_1}, \ \Z_{k_1}{\times} \Z_2,\ \Z_{2^f}{\times}\operatorname{D}_{2{\cdot} 3^e}, \ \Z_{2}^2{:} \Z_{3^e}, \ \Z_2^2{:}\Z_{3^e}, \]
    where $k_1, e,f$ are positive integers.
\end{lemma}

\begin{proof}
    By Lemma~\ref{relationpi}, $G$ is a non-abelian solvable $\mathscr{P}^+_{1}(p)$-group with respect to $(\rho,\tau)$. By Lemma~\ref{solvable-gp-structure}, $H$ is a homomorphic image of $G$. Thus $H$ is generated by $\tilde{\rho}, \tilde{\tau}$, the image of $\rho,\tau$ in $H$. If $H$ is abelian, then $H\cong \Z_{k_1}$ or $\Z_{k_1}{\times}\Z_2$. If $H$ is non-abelian, then by Lemma~\ref{hall23-normal} and Lemma~\ref{hall23-not-normal}, $H$ is isomorphic to  one of the following groups: $\D_{2{\cdot}3^e}$, $\D_{2{\cdot}3^e}{\times}\Z_{2^f}$, $\Z_2^2{:}\Z_{3^e}$, $\S_4{\times}\Z_{3^e}$, and $\Z_2^2{:}\D_{2{\cdot}3^e}$. 
    We only need to rule out the possibility of $\S_4{\times}\Z_{3^e}$ and $\Z_2^2{:}\D_{2{\cdot}3^e}$. 
 These two groups all have a homomorphic image $\S_4$, which is also a homomorphic image of $G$. 
 Let $\bar{\rho}, \bar{\tau}$ be the images of $\rho,\tau$ in $\S_4$. 
 By the proof of Lemma~\ref{relationpi}, $S_4$ is also a $\mathscr{P}_1^+(p)$-group. 
 To be precise, $|\S_4|/\lcm(|\bar{\rho}|, 2{\cdot}|[\bar{\rho}, \bar{\tau}]|)$ is a power of $p$. 
 Using Magma\cite{MAGMA} to run through all generating pairs $(\bar{\rho}, \bar{\tau})$ of $\S_4$ reveals that $p=2$ and hence $\O_p(G)=\O_2(G)=1$. But, $1\neq H'\cap \O_2(H)=\Z_2^2 \neq 1$ is a normal subgroup of $G$, which is a contradiction. 
\end{proof}

The previous lemmas describe the Hall $\{2,3\}$-subgroup of $G$. Next, we investigate when $G$ admits a generating pair $(\rho,\tau)$, and calculate the type of the corresponding map  $\mathscr{M}(G,\rho,\tau)$. We first give a lemma on the order of an element of a metacyclic group, which turns out to be useful in future discussions.

\begin{lemma}\label{lem-meta-order}
 Let $G=\l x\r {:}\l y\r$ be a metacyclic group with $\l x\r=\l [x,y]\r$ and $\gcd(|x|, |y|)=1$. Then for any integer $i$, $|x^iy|=|y|$. In particular, if $|y|=2$, then $x^y=x^{-1}$. 
\end{lemma}

\begin{proof}
 Suppose $k=|y|$ and $x^{y^{-1}}=x^s$. For any integer $i$, it is clear that $k\mid |x^iy|$. 
 The calculation of the $k$-th power of $x^iy$ shows that
 \[(x^iy)^k=x^{i(1+s+s^2+\dots+s^{k-1})}=x^{i(s^k-1)/(s-1)}.\]
 Note that $x=x^{y^{-k}}=x^{s^k}$ which implies that $s^k-1$ is a multiple of $|x|$. If $\gcd(s-1,|x|)\ne 1$, let $r$ be a prime dividing $\gcd(s-1,|x|)$. Then $[x,y^{-1}]=x^{s-1}\in \l x^r\r <\l x\r$, which is a contradiction. Hence $\gcd(s-1,|x|)=1$ and $(s^k-1)/(s-1)$ is a multiple of $|x|$ which implies $(x^iy)^k=1$. This forces $|x^iy|=k=|y|$. Moreover, if $|y|=2$, we have $|xy|=2$. Thus $x^y=yxy=x^{-1}(xy)^2=x^{-1}$ which completes the proof. 
\end{proof}

Now, we are ready to decide when the group $G$ has a generating pair $(\rho,\tau)$. 

\begin{lemma}\label{gensofsolvalbegp}
 Let $G,H,a,b$ be as described in Lemma~\ref{solvable-gp-structure}(ii) and let $\rho,\tau$ be two elements of $G$ with $|\tau|=2$. Then there exist $g\in G$, $\rho_0,\tau_0\in H$ and two integers $0\le i\le |a|-1$, $1\le j\le |b|-1$ such that $\rho^g=a^ib^j\rho_0$, $\tau^g=\tau_0$. Moreover $\l \rho,\tau\r=G$ if and only if $\l \rho_0,\tau_0\r=H$, $a^{\tau_0}=a^{-1}$ and $\gcd(i, |a|)=\gcd(j,|b|)=1$.
\end{lemma}

\begin{proof}
 Recall that, $H$ is a Hall $\{2,3\}$-subgroup of $G$ and thus contains a Sylow $2$-subgroup of $G$. 
 By Sylow theorem, there is an element $g\in G$ such that $\tau^g \in H$. Set $\tau_0=\tau^g$.
 As $G=\l a\r\l b\r H$, there exist two integers $0\le i\le |a|-1$, $1\le j\le |b|-1$ such that $\rho^g=a^ib^j\rho_0$ for some $\rho_0\in H$. 
 It is necessary to show that the condition for $ \rho,\tau$ to be sufficient and necessary ensures the generation of the entire group $G$.

 We first prove the sufficiency. 
 By assumption, \[\l \rho, \tau\r=\l \rho^g,\tau^g\r^{g^{-1}}=\l a^ib^j\rho_0, \tau_0\r^{g^{-1}}.\]  Hence, we only need to show that $\l a^ib^j\rho_0, \tau_0\r=G$. It is easy to see that   \[\l a^ib^j\rho_0, \tau_0\r \l a\r/\l a\r=G/\l a\r.\] The statement holds if $ a\in \l  a^ib^j\rho_0, \tau_0\r$.
 Note that 
  \begin{align*}
    [a^ib^j\rho_0,\tau_0]&=[a^i,\tau_0]^{b^j\rho_0}[b^j\rho_0,\tau_0]=[a^i,\tau_0]^{b^j\rho_0}[b^j,\tau_0]^{\rho_0}[\rho_0,\tau_0] =(a^{-2i})^{b^j\rho_0}[\rho_0,\tau_0].
  \end{align*}
  Since $\Aut(\l a\r)$ is abelian, we have $[\rho_0,\tau_0]\in H'$ acts trivially on $\l a\r$. 
  Recall that $[\rho_0,\tau_0]$ is a $\{2,3\}$-element and $a$ is a $\{2,3\}'$-element. Thus
  \begin{equation}\label{eqn:ordercommutator}
  \l [a^ib^j\rho_0,\tau_0]\r=\l (a^{-2i})^{b^j\rho_0}\r\times \l [\rho_0,\tau_0]\r =\l a\r\times \l [\rho_0,\tau_0]\r.
  \end{equation}
  This gives $a \in \l a^ib^j\rho_0,\tau_0\r$ which implies $\l \rho,\tau\r=G$.

Now we prove the necessity. As $\l \rho,\tau\r=G$, we have
\[\l a^ib^j\rho_0,\tau_0\r=\l \rho^g, \tau^g\r=G^g=G.\] 
 Note that $H$ is a homomorphic image of $G$. 
 We have $\l \rho_0, \tau_0\r=H$. 
 Note that $\l \rho, \tau\r=\l a\r\l b\r H =\l a^ib^j\rho_0,\tau_0\r\leq \l a^i\r\l b^j\r H$. Note that $\l a\r, \l b\r, H$ are Hall subgroups of coprime order, which implies that $\l a^i\r=\l a\r$ and $\l b^j\r=\l b\r$. That is $\gcd(i, |a|)=1$ and $\gcd(j,|b|)=1$. 
 Now, we prove $a^{\tau_0}=a^{-1}$. Consider the coprime action of $\l \tau_0\r$ on $\l a\r$.  
 By \cite[p.198 8.4.2]{kurzweil2004theory}, 
 \[\l a\r =[\l a\r, \l \tau_0\r]{\times }C_{\l a\r}(\l \tau_0\r) =\l [a, \tau_0]\r{\times }C_{\l a\r}(\l \tau_0\r).\]
 If $\l [a, \tau_0]\r=\l a\r$, then, by Lemma~\ref{lem-meta-order}, $a^{\tau_0}=a^{-1}$. If $\l [a, \tau_0]\r<\l a\r$, then $ C_{\l a\r}(\l \tau_0\r) \neq 1$. 
 Denote $A_1=\l [a, \tau_0]\r$, $A_2=C_{\l a\r}(\l \tau_0\r)$, and set $C=C_H(A_2)$. It follows that $C\lhd N_H(A_2)=H$. Note that $C$ also centralizes $\l b\r$. 
 Hence $C\lhd A_2\l b\r H$ and $A_1C \lhd A_1A_2\l b\r H=G$. 
 Denote $\overline{G}=G/A_1C$. 
 Note that $\tau_0\in C_H(C_{\l a\r}(\l \tau_0\r))=C_H(A_2)=C$. 
 Then $\overline{G}$ is generated by the image of $a^ib^j\rho_0$ which means $\overline{G}$ is cyclic and $\overline{G}'=1$.  
  Recall that $[\l a\r, H]=\l a\r$. We have
  \[\overline{G}'\geq [\l a\r, H]A_1C/A_1C=\l a\r A_1C/A_1C\cong \l a\r/(\l a\r\cap A_1C)=\l a\r/A_1\cong A_2\ne 1,\]
  which is a contradiction. 
\end{proof}

Applying Lemma~\ref{lem-meta-order} and Lemma~\ref{gensofsolvalbegp}, we have the following corollary. 

\begin{corollary}\label{chisolvable}
    Let $G,H,a,b$ be  as described in Lemma~\ref{solvable-gp-structure}(ii). Suppose that $\calm=\Map(G,\rho, \tau)$. Then there exist $\rho_0, \tau_0\in H$ and two positive integers $i,j$, such that  $\calm \cong \Map(G, a^ib^j\rho_0, \tau_0)$. 
    Moreover, the type of $\calm$ is $(|a'|{\cdot} |b|{\cdot} |\rho_0|, 2{\cdot}|a|{\cdot} |[\rho_0,\tau_0]|)$,
    where $\l a'\r=C_{\l a\r}(\l b\rho_0\r)$. 
\end{corollary}

\begin{proof}
 By Lemma~\ref{lem:fundamental} {\rm (iii)} and Lemma~\ref{gensofsolvalbegp}, we have $\calm\cong\Map(G, \rho^g, \tau^g)= \Map(G, a^ib^j\rho_0,\tau_0)$ for some $g\in G$, $\rho_0,\tau_0\in H$ and two integers $i,j$ where $\gcd(i,|a|)=\gcd(j,|b|)=1$. 
 Note that $a^ib^j\rho_0\in \l a\r{:}(\l b\r{\times}\l \rho_0\r)=\l a\r{:}(\l b\rho_0\r)$ and $b^j\rho_0$ is a generator of $\l b\rho_0\r$. 
 Consider the coprime action of $\l b\rho_0\r$ on $\l a\r$. 
 By  \cite[p.198 8.4.2]{kurzweil2004theory}, 
 \[ \l a\r =C_{\l a \r}(\l b\rho_0\r){\times} [\l a\r, \l b\rho_0\r].\]
 Assume that $a^i=a'a''$ where $a'\in C_{\l a \r}(\l b\rho_0\r)$, and
\[a''\in [\l a\r, \l b\rho_0\r]=[\l a^i\r, \l b\rho_0\r]=[\l a'a''\r, \l b\rho_0\r]= [\l a''\r, \l b\rho_0\r].\]
Since $\l a\r=\l a^i\r=\l a'a''\r$, we have $\l a'\r=C_{\l a \r}(\l b\rho_0\r)$ and $\l a''\r=[\l a''\r, \l b\rho_0\r]$. 
 By Lemma~\ref{lem-meta-order}, $|a''b^j\rho_0|=|b^j\rho_0|=|b|{\cdot}|\rho_0|$. 
 Thence 
 \[|a^ib^j\rho_0|=|a'a''b^j\rho_0|=\lcm(|a'|, |a''b^j\rho_0|)=|a'|{\cdot}|b|{\cdot}|\rho_0|.\]
 By the same argument used in Lemma~\ref{gensofsolvalbegp}, Equation~\eqref{eqn:ordercommutator} holds which gives 
 \[ |[a^ib^j\rho_0, \tau_0]|=|a|{\cdot} |[\rho_0,\tau_0]|.\]
 Thus the type of $\calm $ is $(|\rho|, 2|[\rho, \tau]|)=(|a'|{\cdot} |b|{\cdot} |\rho_0|, 2{\cdot}|a|{\cdot} |[\rho_0,\tau_0]|)$,
    where $\l a'\r=C_{\l a\r}(\l b\rho_0\r)$.
 \end{proof}

As shown in Corollary~\ref{chisolvable}, given a generating pair $(\rho_0, \tau_0)$ for the subgroup $H$, distinct choices of integers $i,j$ preserve the type  of $\Map(G, a^ib^j\rho_0, \tau_0)$. Therefore, when studying the types or Euler characteristics of such bi-rotary maps $\Map(G, a^ib^j\rho_0, \tau_0)$, we may fix a simple generating pair for $G$, namely $\rho = ab\rho_0$ and $\tau = \tau_0$. We call $(\rho,\tau)$ a \emph{standard rotary pair}.

We now prove Theorem~\ref{thm:main-result-solvable}, together with the type information and the choice of standard rotary pairs of the quotient map $\Map(G,\rho,\tau)$.
    
\begin{theorem}\label{thm:solv}
Let $\calm =\Map(X, x,y)$ be a bi-rotary map of Euler characteristic $-p^n$. Let $G=X/{\rm O}_p(X)$ and let $G$ be non-abelian. If $G$ is solvable, then  $G=\l a\r{:}(\l b\r{\times}H)$ where $H$ is a Hall $\{2,3\}$-subgroup of $G$, and $|a|, |b|, |H|$ are pairwise coprime. Moreover, the possible structure of $H$ and corresponding types $(\bar{k},\bar{m})$ of the quotient map $\calm/\O_p(X)$ are listed in Table~\ref{table solvable}, along with representative standard rotary pairs $(\rho, \tau)$.
\end{theorem}

\begin{proof}
 By definition, $X$ is a $\P{0}{p}$-group with respect to $(x,y)$. 
 Thus, the quotient group $G$ is a non-abelian solvable $\mathscr{P}_1^+(p)$-group with respect to $(\rho, \tau)$ and  $\O_p(G)=1$. 
 Hence, by Lemma~\ref{solvable-gp-structure}, $G=\l a\r{:}(\l b\r{\times}H)$ where $H$ is a Hall $\{2,3\}$-subgroup of $G$, and $|a|, |b|, |H|$ are pairwise coprime. 
 By Lemma~\ref{lem:P1psolvable}, $H$ is isomorphic to one of the following groups: 
\[\Z_{k_1}, \ \Z_{k_1}{\times} \Z_2,\ \Z_{2^f}{\times}\operatorname{D}_{2{\cdot} 3^e}, \ \Z_{2}^2{:} \Z_{3^e}, \ \Z_2^2{:}\Z_{3^e}, \]
where $k_1, e,f$ are positive integers. 
 


 Now, we consider the standard rotary pairs and corresponding types for these cases. By  Corollary~\ref{chisolvable}, we only need to determine the generating pair $(\rho_0,\tau_0)$ of $H$ with $|\tau_0|=2$. Then the type of the quotient map $\calm/\O_p(X)$ is $(|a'|{\cdot} |b|{\cdot} |\rho_0|, 2{\cdot}|a|{\cdot} |[\rho_0,\tau_0]|)$ given by a standard rotary pair $(\rho,\tau)=(ab\rho_0,\tau_0)$. Denote
 \[  k_1=|\rho_0|,\  k_2=|b|, \ m_2=|a|\ \mbox{and}\ m_2'=|C_{\l a\r}(\l a^{-1}\rho\r )|.\]

 First, if $H$ is abelian and  $\tau_0\in \l \rho_0\r$, then $H=\l \rho_0\r$ and $\tau_0=\rho_0^{|\rho_0|}=\rho_0^{k_1/2}$. 
 By Corollary~\ref{chisolvable}, the type of $\calm/\O_p(X)$ is $(k_1k_2m_2', 2m_2)$.
 This gives line $1$ in Table~\ref{table solvable}. 
 If $H$ is abelian and   $\tau_0\notin \l\rho_0\r$, then  $H=\l \rho_0\r{\times}\l \tau_0\r$. 
 Again, by Corollary~\ref{chisolvable}, the type of $\calm/\O_p(X)$ is $(k_1k_2m_2', 2m_2)$. 
 Note that $\l \rho_0,[\rho_0,\tau_0]\r=\l \rho_0\r$ is an index two subgroup of $H$, and so is its preimage in $X$. By Lemma~\ref{lem:fundamental} {\rm (1)}, $\calm$ is orientable, and $p=2$.
 This gives line $2$ in Table~\ref{table solvable}. 

 Now, if $H=\l c\r{:}\l d\r\cong \D_{2{\cdot}3^e}$, then $H'\cap \O_3(H)=\Z_{3^e}\neq 1$ is also a normal subgroup of $G$ as $H'$ centralizes $\l a\r$. This gives $p\neq 3$. 
 Since $(\rho_0,\tau_0)$ is a generating pair of $H$ with $|\tau_0|=2$,   $|\rho_0|=3^e$ or $2$. 
 If $|\rho_0|=3^e$, without loss of generality, assume $(\rho_0,\tau_0)=(c,d)$. 
 The corresponding type of  $\calm/\O_p(X)$ is $(3^ek_2m_2', 2{\cdot}3^em_2)$. 
 By the same argument used in line $2$, $\calm$ is orientable, and $p=2$.
  This gives line $3$ in Table~\ref{table solvable}. 
 If $|\rho_0|=2$, without loss of generality, assume $(\rho_0,\tau_0)=(cd,d)$. 
 The corresponding type of  $\calm/\O_p(X)$ is $(2k_2m_2', 2{\cdot}3^em_2)$. Note that $\l a^{-1}\rho\r=\l b\rho_0\r=\l b\r{\times}\l cd\r$ and $c\in H'$ centralizes $\l a\r$. We have 
 \[m_2'=|C_{\l a\r}(\l a^{-1}\rho\r) |=|C_{\l a\r}(\l b\r{\times}\l cd\r) |\leq |C_{\l a\r}(\l cd\r)|=|C_{\l a\r}(\l d\r)|=1.\] 
 Hence $m_2'=1$, which gives line $4$ in Table~\ref{table solvable}.

 If $H=\l d_1\r {\times}\l c\r{:}\l d_2\r=\Z_{2^f}{\times}\D_{2{\cdot}3^e}$, then $H/H'\cong \Z_2^2$ is not cyclic. As $H$ is a homomorphic image of $X$, $X/X'$ is not cyclic either. Hence $p=2$, otherwise, $X/X'=\l x, X'\r /X'$ is cyclic, which is a contradiction. By a similar argument as above,   without loss of generality, assume $(\rho_0,\tau_0)=(cd_1,d_2)$ or $(cd_1d_2,d_2)$. The corresponding types of  $\calm/\O_p(X)$ are $(2^f3^ek_2m_2', 2{\cdot}3^em_2)$ or  $(2^fk_2m_2', 2{\cdot}3^em_2)$, respectively.  This gives lines $5$ and $6$ in Table~\ref{table solvable}. 

 Finally, if  $H=(\l d_1\r {\times}\l d_2\r){:}\l c\r=\Z_{2^2}{:}\Z_{3^e}$, then $H'\cap \O_2(H)=\Z_2^2$ is also a normal subgroup of $G$. Hence $p\neq 2$. By Lemma~\ref{gensofsolvalbegp}, we know that $a^\tau_0=a^{-1}$. On the other hand, $\tau_0\in H'$ should act trivially on $\l a\r$. As $|a|$ is odd, we have $a=1$. Without loss of generality, we may assume $(\rho_0, \tau_0)=(d, c_1)$.  The corresponding type of  $\calm/\O_p(X)$ is $(3^ek_2, 4)$. This gives line $7$ in Table~\ref{table solvable}. 
\end{proof}

\section{Non-Solvable Groups}\label{sec:non-solvable}

In this section, we prove Theorem~\ref{thm:main-non-solvable}. Following a similar approach to the preceding section, we define $\Map(X,x,y)$ as a bi-rotary map with Euler characteristic $-p^n$. Let $G = X/\O_p(X)$, and set $\rho = x\O_p(X)$ and $\tau = y\O_p(X)$. Consequently, $G$ is a non-solvable $\mathscr{P}_1^+(p)$-group with respect to the pair $(\rho, \tau)$, and  $\O_p(G) = 1$. 
We first present a lemma addressing the case where 
$G$ is a simple group.


\begin{lemma}\label{lem:simple-gps1}
Let $G$ be a non-abelian simple group, and let $p$ be a prime. 
\begin{itemize}
    \item [\rm (i)] If $G$ is a $\mathscr{P}_2(p)$-group,
then $G$ is one of the following groups:
$\PSL(2,p^t)$, $\PSL(2,r)$ for some prime $r$, ${\rm Sz}(q)$, $\A_7$, or $\J_1$.
    \item [\rm (ii)] If $G$ is a $\mathscr{P}_1(p)$-group, then $G=\PSL(2,q)$ where $q\ge 5$ is a prime, or $q=p^t\ge 4$ for some $t$.
\end{itemize}

\end{lemma}
\begin{proof}
We prove part {\rm (i)} by analyzing each family of finite simple groups. 

Assume that $G=\A_n$.
If $n\geqslant 8$, then a Sylow 3-subgroup of $G$ is not cyclic, and a Sylow 2-subgroup of $G$ is neither cyclic nor dihedral, which is a contradiction.
Thus $n\leqslant 7$. 
As $\A_5\cong\PSL(2,5)$ and $\A_6\cong\PSL(2,9)$, part {\rm (i)} holds for alternating groups.

If $G$ is a sporadic simple group, then $G=\J_1$ by inspecting the sporadic simple group in the Atlas.

Assume that $G$ is an exceptional simple group of Lie type of characteristic $r$.
Then Sylow $r$-subgroups of $G$ are neither cyclic nor dihedral.
Thus $p=r$ and any other Sylow subgroups are cyclic or dihedral.
It follows that $G={\rm Sz}(q)$ with $p=2$.

Finally, suppose that $G$ is a classical simple group of Lie type of dimension $n$ and characteristic $r$.
Suppose that $n\geqslant3$, and $G\not=\PSL(3,2)\cong \PSL(2,7)$.
Then a Sylow $p$-subgroup of $G$ is neither cyclic nor dihedral, implying that $p=r$.
If $p=2$, then $G$ contains a subgroup $H$ which is isomorphic to $\PSL(3,4)$ or $\PSL(4,2)$, and a Sylow 3-subgroup of $H$ is not cyclic, which is a contradiction.
If $p$ is odd, then a Sylow $2$-subgroup of $G$ is neither cyclic nor dihedral, again a contradiction.
Therefore, it must be that $n=2$.
If $G=\PSL_2(q)$ with $q=r^t$ odd and $t>1$, then a Sylow $r$-subgroup of $G$ is neither cyclic nor dihedral, and so $p=r$.
We thus conclude that $G=\PSL(2,r)$ or $\PSL(2,p^t)$.

By Lemma~\ref{relationpi}, to prove part (ii), we only need to examine the groups listed in part (i).
We shall prove that ${\rm Sz}(q)$, $A_7$ and $\J_1$ are not $\mathscr{P}_1(p)$-groups. 

Let $|G|=p^n\lcm(|H_1|,|H_2|)$ where $H_1$ and $H_2$
are cyclic or dihedral. 
If $G=\A_7$, then by Lemma~\ref{relationpi}, we have $p=3$ as a Sylow $3$-subgroup of $\A_7$ is not cyclic. 
We then have that \[7{\cdot} 5{\cdot} 3^2{\cdot} 2^3=|\A_7|=3^n\cdot\lcm(|H_1|,|H_2|),\] 
so $7\mid |H_1|\mbox{ or }|H_2|$. 
Without loss of generality, assume $7\mid |H_1|$ and let $x\in H_1$ be an element of order $7$.
Since $H_1$ is cyclic or dihedral,
we have $H_1\le N_G(\l x\r)$.
Notice that $N_G(\l x\r)\cong \mathrm{Z}_7{:}\mathrm{Z}_3$. 
We have $H_1$ cannot be dihedral, which implies that $H_1=\l x\r\cong\mathrm{Z}_7$.
Therefore $2^3\cdot 5\mid |H_2|$.
Since $H_2$ is either cyclic or dihedral, 
$H_2$ has an element of order $20$, which is impossible for $\A_7$.
Assume $G={\rm Sz}(q)$, where $q=2^{2t+1}$ for some integer $t$. Then $p=2$ as Sylow $2$-subgroup of ${\rm Sz}(q)$ is not cyclic nor dihedral. Therefore, 
\begin{equation}\label{equ:suz}
    q^2(q^2 + 1)(q-1)=2^n\lcm(|H_1|,|H_2|).
\end{equation}
By \cite[ XI, Theorem~3.10]{huppert1982Finite}, for any element in ${\rm Sz}(2^{2t+1})$, its order divides $4\mbox{ or }2^{2t+1}-1$ or $2^{2t+1}\pm 2^{t+1}+1$. 
Hence, ${\rm Sz}(q)$ does not satisfy Equation (\ref{equ:suz}) for any cyclic or dihedral subgroups $H_1$ and $H_2$.
If $G=\J_1$, then $p=2$ as Sylow $2$-subgroup of $\J_1$ is neither cyclic nor dihedral. 
Then we have 
\begin{equation*}
    2^3{\cdot} 3 {\cdot} 5 {\cdot} 7 {\cdot} 11 {\cdot} 19=2^n\lcm(|H_1|,|H_2|).
\end{equation*} 
According to \cite{wilsonATLAS}, the order of any element in the group $\J_1$ must be one of $1,2,3,5,6,7,$ $10,11,15,19$. 
Consequently, there do not exist suitable subgroups $H_1$ and $H_2$ in $\J_1$ that satisfy the aforementioned equation. This implies that $\J_1$ is not a $\mathscr{P}_1(p)$-group.
\end{proof}

By applying Lemma~\ref{lem:simple-gps1}, we can establish a characterization of the structure of non-solvable $\mathscr{P}_1^+(p)$-groups $G$ with $\O_p(G)=1$ in the following lemma.

\begin{lemma}\label{lem:almostsimple}
Let $G$ be a non-solvable $\mathscr{P}_1^+(p)$-group with respect to $(\rho,\tau)$, where $p$ is a prime. Assume $\O_p(G)=1$. Set $R=\rad(G)$ and  $D=G^{(\infty)}$. Then the following statements hold:
\begin{itemize}
    \item [\rm (i)] The factor group $G/R$ is almost simple. Specifically, $G/R\cong \PSL(2,q).\Z_f$ for some integer $f$, where either $q\ge 5$ is a prime or $q=p^t\ge 4$ for some integer $t$.
    \item [\rm (ii)] The intersection $R\cap D$ is trivial. It follows that $D\cong \PSL(2,q)$ and $G=(R\times D).\Z_f$ where $q$ and $f$ are described in {\rm (i)}.
\end{itemize}

\end{lemma}

\begin{proof}
Let $S=\mathrm{soc}(G/R)$.
Since $G$ is a $\mathscr{P}_1^+(p)$-group, we have that $S$
is a $\mathscr{P}_1(p)$-group by Lemma~\ref{relationpi} {\rm (i)} and Lemma~\ref{lem:p1}.
Note that $S$ is a direct product of simple groups $T_i$. By Lemma~\ref{lem:p1} {\rm (ii)} and Lemma~\ref{lem:simple-gps1}, $$S\cong\PSL(2,q_1)\times\dots\times\PSL(2,q_\ell),$$ where $q_j$ is either a prime or of the form $q_j=p^{t_j}$ for some integer $t_j$,  for all $j\in\{1,\ldots,\ell\}$. 
If $\ell>1$, then the Sylow $3$-subgroup of $S$ is not cyclic since $3$ divides $ |\PSL(2,q_j)|$ for all $j$. Similarly,  the Sylow $2$-subgroup of $S$ is neither cyclic nor dihedral. This contradicts the fact that $G$ is a $\mathscr{P}_2(p)$-group, by Lemma~\ref{relationpi}. 
Hence $\ell=1$ and $S\cong \PSL(2,q)$ is the unique minimal normal subgroup of $G/R$, where $q=p^t$ or $q$ is a prime. In particular,  $G/R$ is almost simple. Set $a=\rho R$ and $b=\tau R$. Then   $G/R=\l a,b\r$ as $G=\l \rho, \tau\r$.
If $p$ is odd, by the definition of $\P{1}{p}$-group, $G=\l [\rho, \tau], \rho\r$. It follows that  $G/R=\l[a,b],a \r$. Since $\Out(S)$ is abelian, we have $[a,b]\in (G/R)'=S$, and $(G/R)/S$ is cyclic. This gives $G/R=S.\Z_f$ for some integer $f$.
If $p=2$, then $S\cong\PSL(2,2^t)$ or $S\cong\PSL(2,r)$ for some prime $r\neq 2$. In each case, $\Out(S)$ is cyclic and we also have that $G/R=S.\Z_f$ for some integer $f$. Thus, part (i) holds.

Now, let $M=R\cap D$. As  $\O_p(D)\le\O_p(G)=1$, the Fitting subgroup $\operatorname{F}(D)$ of $D$ satisfies
\[\operatorname{F}(D)=\O_{p_1}(D)\times\cdots\times\O_{p_i}(D)\] 
for some primes $p_1,\ldots,p_i$ that are coprime to $p$.
By N-C Lemma, 
\[D/C_D(\operatorname{F}(D))\lesssim \Aut(\O_{p_1}(D))\times\cdots\times\Aut(\O_{p_i}(D)).\]
Since $G$ is a $\mathscr{P}_1(p)$-group, we have that $D$ is a $\mathscr{P}_2(p)$-group by Lemma~\ref{relationpi} and Lemma~\ref{lem:p2}, so $\O_{p_1}(D),\ldots,\O_{p_i}(D)$ are all cyclic or dihedral.
Then \[\Aut(\O_{p_1}(D))\times\cdots\times\Aut(\O_{p_i}(D))\] is solvable, which implies that $D/C_D(\operatorname{F}(D))$ is solvable. Since $D$ is perfect, we have $D=C_D(\operatorname{F}(D))$ and $\operatorname{F}(D)\le Z(D)$. Note that $Z(D)$ is a solvable normal subgroup of $G$. Hence,  $Z(D)\le R$, we have that $\operatorname{F}(D)\le Z(D)\le R\cap D=M$. Since $M\lhd D$, the Fitting subgroup $\operatorname{F}(M)$ of $M$ is a subgroup of $\operatorname{F}(D)$. Thus 
\[ M =C_M(\operatorname{F}(D))\leq C_M(\operatorname{F}(M))\leq \operatorname{F}(M)\leq M,\]
since $M\leq R$ is solvable. This gives \[\operatorname{F}(M)= M \leq \operatorname{F}(D)\leq M. \]
It follows that $M=\operatorname{F}(D)=Z(D)$. Note that 
\[D/M\cong DR/R=(G/R)^{(\infty)}\cong S\cong \PSL(2,q),\]
and $D$ is a perfect group. By part (i), $D$ is a quasi-simple group with $D/Z(D)\cong \PSL(2,q)$, where either $q\ge 5$ is a prime or $q=p^t\ge 4$ for some integer $t$.
By \cite[Section 3.3.6]{wilson-gtm251} and $\O_p(D)=1$, either $M=1$ and $D\cong \PSL(2,q)$,  or $M\cong \Z_2$, and $D\cong \SL(2,q)$. 
If the latter case holds, $M$ is the unique subgroup of order $2$ in $D$, and $q\geq 5$ is an odd prime. Moreover, the Sylow subgroup of $\SL(2,q)$ is neither cyclic nor dihedral, which implies $p=2$. But \[M=\O_2(D)=\O_p(D)\leq \O_p(G)=1,\] leading to a contradiction. Therefore, we have $1=M=D\cap R$,
and  $RD=R\times D$.
Furthermore, as  $G/RD\cong (G/R)/(RD/R)\cong\Z_f$, we have $G=(R\times D).\Z_f$ and part (ii) holds.
\end{proof}

For a $\mathscr{P}_1^+(p)$-group $G$ with $\O_p(G)=1$, Lemma~\ref{lem:almostsimple} tells us that if $G$ is non-solvable, then $G$ has precisely one non-abelian simple composition factor, namely $\PSL(2,q)$. In particular, $G/\rad(G)$ is an almost simple $\mathscr{P}_1^+(p)$-group with socle $\PSL(2, q)$, while $G/G^{(\infty)}$ is a solvable $\mathscr{P}_1^+(p)$-group respectively. 
To understand the group structure of $G$, a natural approach is to first determine the structure of the almost simple $\mathscr{P}_1^+(p)$-group $G/\rad(G)$, and then use the result of Theorem~\ref{thm:main-result-solvable} to further restrict the structure of $G$. We establish a useful property of $\PSigL(2,q)$ in Lemma~\ref{lem:glcentralizer}. This will be employed to characterize the structure of $G/\rad(G)$ in Lemma~\ref{lem:M/R}. Using these results and Theorem~\ref{thm:main-result-solvable}, we further analyze the structure of $G$ in Lemma~\ref{lem:M/DP} and Corollary~\ref{cor:restriction-M/D}.



\begin{lemma}\label{lem:glcentralizer}
    Let $q$ be a prime power such that $q\ge 4$, and let $G=\mathrm{P\Sigma L}(2,q)$.
    If $g\in \mathrm{PSL}(2,q)$ and $|g|=\frac{q\pm 1}{\gcd(2,q-1)}$,
    then $C_G(\l g\r)=\l g\r$.
\end{lemma}

\begin{proof}
    Let $d=\gcd(2,q-1)$, and let $q=p^t$ for some integer $t\geq 1$. 
    By \cite[II, Theorem~8.3-8.4]{huppert1967Endliche}, $N_{\PSL(2,q)}(\l g\r)\cong \D_{2 |g|} $.
    Also, by \cite[II, Theorem~8.5]{huppert1967Endliche}, all the cyclic subgroups of order $| g|$ in $\PSL(2,q)$ are conjugate. Hence $$C_{\PGaL(2,q)}(\l g\r)\leq N_{\PGaL(2,q)}(\l g\r)= N_{\PSL(2,q)}(\l g\r). (\Z_2\times \Z_{t}).$$
    If $|g|=\frac{q-1}{d}$, let 
    \[N=\left\l \begin{pmatrix}
        \omega& \\ & 1
    \end{pmatrix},\begin{pmatrix}
        1& \\ & \omega
    \end{pmatrix},\begin{pmatrix}
        & 1\\ 1& 
    \end{pmatrix},\gamma \right\r \leq \GaL(2,q). \]
 where $\omega$ is a generator of $\mathbb{F}_q^{\times}$
 and $\gamma$ is the field automorphism.
 Then the image $\bar{N}$ of $N$ under the natural homomorphism $\GaL(2,q)$ to $\PGaL(2,q)$ includes a normal subgroup of order $\frac{q-1}{2}$ contained in $\PSL(2,q)$.
 Therefore, $\bar{N}$ is conjugate to $N_{\PGaL(2,q)}(\l g\r)$.
 A simple computation shows that $C_{\PGaL(2,q)}(\l g\r)\le \PGL(2,q)$ and 
 \[ C_{\PSigL(2,q)}(\l g\r)\le \PGL(2,q)\cap \PSigL(2,q)=\PSL(2, q).\]
 This implies that $C_{\PSigL(2,q)}(\l g\r)=\l g\r$.
If $|g|=\frac{q+1}{d}$, let 
    \[N=\left\l \begin{pmatrix}
        \omega^{q-1}& \\ & 1
    \end{pmatrix},\begin{pmatrix}
        1& \\ & \omega^{q-1}
    \end{pmatrix},\gamma \right\r \leq \Gamma \mathrm{U}(2,q). \]
 where $\omega$ is a generator of $\mathbb{F}_{q^2}^{\times}$
 and $\gamma$ is the field automorphism.
 Then the image $\bar{N}$ of $N$ under the natural homomorphism $\Gamma \mathrm{U}(2,q)$ to $\PGaL(2,q)\cong\PGaU(2,q)$ includes a normal subgroup of order $\frac{q+1}{2}$ contained in $\PSL(2,q)$.
 Therefore, $\bar{N}$ is conjugate to $N_{\PGaL(2,q)}(\l g\r)$.
 Analogously, we have that $C_{\PSigL(2,q)}(\l g\r)=\l g\r$. This completes the proof.
\end{proof}

\begin{lemma}\label{lem:M/R}
    Let $G$ be an almost simple $\mathscr{P}_1^+(p)$-group with respect to $(\rho,\tau)$. Then one of the following holds:
    \begin{itemize}
        \item [\rm(i)] $q=p^t\ge 4$ and $G\cong \mathrm{PSL}(2,q)$ or $\mathrm{PSL}(2,2^t).\operatorname{Z}_2$;
        \item [\rm (ii)] $p>2$ and $G\cong \mathrm{PSL}(2,q)$
        where $q=2p^t\pm 1 \ge 5$ is a prime;
        \item [\rm (iii)] $p=2$ and $G\cong \mathrm{PSL}(2,q)$ or $ \mathrm{PGL}(2,q)$ where $q\ge 5$ is a Mersenne prime or a Fermat prime.
    \end{itemize}
\end{lemma}
\begin{proof}
    By Lemma~\ref{lem:almostsimple}, $G \cong  \PSL(2,q).\Z_f$ for some integer $f$.
    Let $\omega=[\rho,\tau]=\tau^\rho\tau \in G'\cong \PSL(2,q)$. 
    By Lemma~\ref{relationpi} {\rm (i)} and Lemma~\ref{lem:p1}, 
    $|G|=p^n\lcm(|\l \rho\r|,|\l \tau^\rho, \tau \r|)$ for some integer $n$. 
    Note that $|\l \tau^\rho, \tau \r|=2|\omega|$.
    Therefore, 
    \begin{equation}\label{equ:psl}
        \frac{q(q+1)(q-1)f}{d I}=p^n,
    \end{equation}
    where $d=\gcd(2,q-1)$, $I=\lcm(|\rho|, 2|\omega|)$. Keep in mind that $q, (q+1)/2, (q-1)/2$ are pairwisely coprime. 
    
    If $p\nmid q$, then by Lemma~\ref{lem:almostsimple},
    we have that $q\ge 5$ is a prime. This implies that $f\le 2$
    and $d=2$.
    As $\omega\in\PSL(2,q)$, 
    we have $|\omega|$ divides one of the integers  $ q, (q+1)/2,$ or $(q-1)/2$.
    If $f=1$, then $G\cong\PSL(2,q)$, which also  implies that
    $|\rho| $ divides one of the integers  $ q, \frac{q+1}{2},$  $\frac{q-1}{2}$.
    Therefore, $I$ divides one of the integers $2q, q+1, q-1,$$q(q+1), q(q-1) $ or $(q+1)(q-1)/2$. 
    If $I$ divides $2q, q+1$ or $q-1$, then the left side of Equation \eqref{equ:psl} cannot be a prime power.
    If $I$ divides $(q+1)(q-1)/2$, then the left side of Equation \eqref{equ:psl} is a multiple of $q$, which is not divisible by $p$.
    Hence, $I$ divides $q(q+1)$ or $q(q-1) $, and 
    we have that either $\frac{q-1}{2}=p^n$ or $\frac{q+1}{2}=p^n$.
    Therefore, $q=2p^n\pm 1$ and either part (ii) or part (iii) holds.  
    In particular, if $p=2$, then $q$ is a Mersenne prime or a Fermat prime and part (iii) holds.
    If $f=2$, then $G\cong\PGL(2,q)$ which implies
    that $|\rho|$ divides one of the integers $q, q+1,\mbox{ or }q-1$. 
    Therefore, we also have $I$ divides one of the integers $2q, q+1, q-1,$$q(q+1), q(q-1) $ or $(q+1)(q-1)/2$. 
    Using the same analysis as above, we have $q\pm 1=p^n$. 
    Recall that $q\ge 5$ is an odd prime.
    The number $p^n$ is even, which implies that $p=2$ and 
    $q=2^n\pm 1$, which is either a Mersenne prime or a Fermat prime. In this case, part (iii) holds. 
    
    Now we go to the case $p\mid q=p^t$.
    If $p$ is odd, then $G=\l \rho,\omega\r$ and $d=2$.
    Since $\omega\in G'\cong\PSL(2,q)$, we have that $G=\PSL(2,q)\l \rho\r$. 
    By Equation \eqref{equ:psl}, 
    we have
    \begin{equation}
        \label{equ:psl1}
        \frac{(q+1)(q-1)f_{p'}}{2I_{p'}}=1
    \end{equation}
    As $\omega, \rho^f\in \PSL(2,q)$,
    we have that 
    $|\omega|$ divides one of integers $p,\frac{q+1}{2}\mbox{ or }\frac{q-1}{2}$, and 
     $|\rho|$ divides one of integers $pf, \frac{(q+1)f}{2}\mbox{ or }\frac{(q-1)f}{2}$.
    Note that $I_{p'}=\lcm(|\rho|_{p'},(2|\omega|)_{p'})$.
    Therefore, 
    $I_{p'}$ divides one of integers\[f_{p'},\frac{(q+1)f_{p'}}{2},\frac{(q-1)f_{p'}}{2}, \mbox{ or } \frac{(q+1)(q-1)f_{p'}}{2}.\]
    This implies that $I_{p'}=\frac{(q+1)(q-1)f_{p'}}{2}=2|\rho|_{p'} |\omega|_{p'}$, 
    where $|\rho|_{p'}=\frac{(q\pm 1)f_{p'}}{2}$ and $2|\omega|_{p'}=(q\mp 1)$.
    If $2\mid f_{p'}$, then $2$ is a common divisor of $|\rho|_{p'}$ and $2|\omega|_{p'}$ which leads to the contradiction $I_{p'}\leq \frac{2|\rho|_{p'} |\omega|_{p'}}{2}=\frac{I_{p'}}{2}< I_{p'}$. Hence $f$ is odd and $\rho\in \mathrm{P\Sigma L}(2, q)$. By Lemma~\ref{lem:glcentralizer}, $C_G(\l \rho^f\r) =\l \rho^f\r$. 
    As $\rho\in C_G(\l\rho^f\r)=\l \rho^f\r $, we have $f=1$ and part (i) holds. 
    Now suppose that $p=2$. Then $d=1$, and Equation \eqref{equ:psl} implies
     \begin{equation}
        \label{equ:psl2}
        (q+1)(q-1)f_{2'}=I_{2'}
    \end{equation}
    Let $\l \rho'\r =\l \rho\r \cap \PSL(2,q)$. Note that $\l \rho, \tau\r=G\leq {\rm P\Gamma L}(2, q)={\rm P\Sigma L}(2, q)=\PSL(2,q).\Z_t$ and $\tau^2=1$. We have $|\rho|_{2'}=f_{2'}|\rho'|_{2'}$. 
    As $\omega, \rho'\in \PSL(2,q)$,
    we have that 
   both $|\omega|$ and $|\rho'|$ divides one of the integers $2,q+1, q-1$.
   Since $I_{2'}=\lcm(|\rho|_{2'},(2|\omega|)_{2'}) =\lcm(f_{2'}|\rho'|_{2'},|\omega|_{2'}) =(q+1)(q-1)f_{2'}$, we have 
   $|\rho'|_{2'}= q\pm 1$ and $|\omega|_{2'}= q\mp 1$. Note that $\rho\in C_G(\l \rho' \r)=C_{{\rm P \Sigma L}(2, q)}(\l \rho'\r)$. By Lemma~\ref{lem:glcentralizer}, $\rho \in \l \rho'\r $ which gives $\rho \in \PSL(2,q)$. Therefore $G/\PSL(2,q)\lesssim \l \tau\r\cong \Z_2$. This gives $f\leq 2$ and part (i) holds. We conclude the proof.
\end{proof}

\begin{lemma}
    \label{lem:M/DP}
    Let $G$ be a non-solvable $\mathscr{P}_1^+(p)$-group with respect to $(\rho,\tau)$ and let $\O_p(G)=1$.
    Let $D=G^{(\infty)}$ and $R=\rad(G)$.
    Then one of the following holds:
    \begin{itemize}
        \item [\rm (i)] $p$ is odd and $G/D\cong R$ is a cyclic group of odd order;
        \item [\rm (ii)] $p=2$ and $G/D\cong R$;
        \item [\rm (iii)] $p=2$ and $G/D\cong R\times \Z_2$;
        \item [\rm (iv)] $p=2$, $G/D\cong R.\Z_2$ and $\O_2(G/D)=1$.
    \end{itemize}
\end{lemma}
\begin{proof}
    By Lemma~\ref{lem:almostsimple}, we have $G=(R\times D). \Z_f$ and $G/D=\PSL(2,q).\Z_F$ is an almost simple group. Note that $G/D$ is a $\mathscr{P}_1^+(p)$-group. By Lemma~\ref{lem:M/R}, $f\leq 2$ when $p=2$ and $f=1$ when $p>2$. 

    If $p$ is odd, then $G= R\times D$, where 
    $D=\PSL(2,q)$ for some $q$.
    Since $G$ is a $\mathscr{P}_2^+(p)$-group with respect to $(\rho,\tau)$, we have that 
    $|R|_{p'}$ is coprime to $|D|_{p'}$.
    By the fact that $2\mid |D|$, $|R|$ is odd. Moreover, since $R\cong G/D$, we have that $R$ is also a $\mathscr{P}_2^+(p)$-group with respect to $(\rho D,\tau D)$.
    This implies that $\tau \in D$ and $R$ is a cyclic group of odd order, and we have {\rm (i)}.

    If $p=2$, then
    $G/D\cong R\mbox{ or }R.\Z_2$.
    Let $O:=\O_2(G/D)$.
    Since $\O_2(R)=1,$ we have that $O\cap (RD/D)\lesssim \O_2(R)=1$.
    Then either $O=1\mbox{ or }\Z_2$.
    If $O=1$, then we have {\rm (ii)} and {\rm (iv)}.
    If $O=\Z_2$, then we have $G/D\cong R\times O\cong R\times \Z_2$, which is {\rm (iii)}.
\end{proof}

Theorem~\ref{thm:main-non-solvable} now follows from Lemmas \ref{lem:almostsimple}, \ref{lem:M/R}, and \ref{lem:M/DP}.

\begin{proof}[Proof of Theorem~\ref{thm:main-non-solvable}]
  It follows immediately from Lemma~\ref{lem:almostsimple} and Lemma~\ref{lem:M/R} that $G=(R\times D).\Z_f$, $f\leq 2$ and $D\cong\PSL(2,q)$ for some $q$. We now prove the ``moreover'' part of the theorem.
    
    If $p$ is odd, then according to Lemma~\ref{lem:M/R}, we have $f=1$ and $q=2p^t\pm 1$ or $q=p^t$ for some $t$.
    By Lemma~\ref{lem:M/DP}, $R$ is a cyclic group of odd order, and  {\rm (i)} holds.

    Now, suppose $p=2$. Then by Lemma~\ref{lem:M/R}, either $q=2^t$ or $q$ is a Mersenne prime or a Fermat prime. That is $q\in \num$. If $f=1$, then {\rm (ii)} holds
    If $f=2$, then by Lemma~\ref{lem:M/DP}, $G/R\cong\PSL(2,q).\Z_2$, and either $G/D=R\times \Z_2$ or $\O_2(G/D)=1$. If $\O_2(G/D)=1$, then {\rm (iv)} holds.  If $G/D=R\times \Z_2$, then $G$ has a normal subgroup $N$ such that $N\cap R=1$ and $G/N= R $. This gives $G=R\times N=R\times (\PSL(2,q). \Z_2)$ and $q\in \num$. Thus,  {\rm (iii)} holds. \end{proof}

\begin{corollary}\label{cor:restriction-M/D}
    Using the notation in Theorem~\ref{thm:main-non-solvable},
    we have that: 
    \begin{itemize}
        \item [\rm (i)] if $p>2$, then $G=D\times \Z_\ell$ and $\gcd(\ell,|D|)=1$;
        \item [\rm (ii)] if $p=2$, then $G/D=\l a\r{:}(\l b\r \times H)$ where $H=\Z_{2^e}$ or $\Z_{2^e}\times \Z_2$ and $|a|, |b|, |D|$ are pairwise coprime.  
    \end{itemize}
\end{corollary}
\begin{proof}
    If $p>2$, then it follows from Theorem~\ref{thm:main-non-solvable} that $G=D\times\Z_\ell$.
    Note that $p\nmid \ell$ as $\O_p(G)=1$.
    Since $G$ is a $\mathscr{P}_2(p)$, we have that $\gcd(\ell,|D|)=1$.

    If $p=2$, then it follows from Theorem~\ref{thm:main-non-solvable} that $G=R\times D$ or $(R\times D).\Z_2$
    and $D\cong\PSL(2,q)$. Note that $G/D$ is a $\mathscr{P}_2^+(2)$-group and $R\times D$ is a $\mathscr{P}_2(2)$-group.
    Since $3\mid|D|$, we have that $3\nmid |R|$, which implies that $3\nmid|G/D|$. 
    Lemma~\ref{lem:M/DP} shows that either $\O_2(G/D)=1$ or $G/D\cong R\times\Z_2$ where $\O_2(R)=1$.
    If $\O_2(G/D)=1$, then by Lemma~\ref{solvable-gp-structure}, $G/D=\l a\r{:}(\l b\r\times H)$, where $|a|, |b|,6$ are pairwisely coprime and $H$ is a Hall $\{2,3\}$-group of $G/D$. In particular, $H$ is a $2$-group. That is $H=\Z_{2^e}$ or $\Z_{2^e}\times \Z_2$. Moreover, as $R\times D$ is a $\mathscr{P}_2(2)$-group, we have $|a|, |b|,|D|$ are pairwisely coprime. If $G/D\cong R\times\Z_2$ and $\O_2(R)=1$, then $R\cong G/(D.\Z_2)$ is also a  $\mathscr{P}_2^+(2)$-group. Again,  by Theorem~\ref{solvable-gp-structure},  we have  $R=\l a\r{:}(\l b\r\times H')$, where $|a|, |b|,|D|$ are pairwisely coprime and $H'=\Z_{2^e}$ or $\Z_{2^e}\times \Z_2$. Therefore $G/D=\l a\r{:}(\l b\r \times H') \times \Z_2$. Set $H=H'\times\Z_2$. Then $H\cong (G/D)/(\l a\r{:}\l b\r)$ is also $2$-generated which implies $H=\Z_{2^e}$ or $\Z_{2^e}\times \Z_2$. This completes the proof.    
\end{proof}

\section{Examples of Bi-rotary maps}\label{sec:example}

In this section, we provide explicit examples of bi-rotary maps with Euler characteristics that are negative prime powers corresponding to each case in Theorems~\ref{thm:main-result-solvable} and \ref{thm:main-non-solvable}. In particular, for Theorem~\ref{thm:main-result-solvable} (i.e., the solvable case), we provide an infinite family for each case listed in Table~\ref{table solvable}. For comprehensive details, see the next subsections.

\subsection{\texorpdfstring{Examples of bi-rotary maps with $\chi=-p^n$: solvable cases}{Examples of solvable cases }}
In this subsection, we provide infinite families of examples corresponding to each of the seven lines listed in Table~\ref{table solvable}.

\begin{example}
 For any positive integer $f$, let $m_2=(23^{6f-5}+3)/2$. Thus $m_2/13$ is an integer which is coprime to $2{\cdot}3{\cdot}13{\cdot}23$. Now, set $X=(\Z_{13}{\times}\Z_{m_2/13}){:}\Z_6=(\l a_1\r{\times}\l a_2\r){:}\l x\r$, where $a_1^x=a_1^4$, $a_2^x=a_2^{-1}$. Set $y=a_1a_2x^3$. Then the bi-rotary map $\Map(X,x,y)$ is of Euler characteristic
 \[\chi =6m_2(\frac{1}{6}-\frac{1}{2}+\frac{1}{2m_2})=3-2m_2=-23^{6f+3}.\]
 Moreover, we have $\O_{23}(X)=1$ and the Hall $\{2,3\}$-subgroup of $X/\O_{23}(X)$ is isomorphic to $\l \bar{x}\r \cong \Z_6$. 
 This gives examples of Line~1 in Table~\ref{table solvable}. \qed
\end{example}

\begin{example}
 For any positive integer $f$, let $m_2=(2^{42f-1}+10)/6$. Then $m_2/7$ is an integer which is coprime to $42$. Now let $$X=(\Z_7 {\times} \Z_{m_2/7}){:(\Z_3\times \D_8)}=(\l a_1\r \times \l a_2\r ) {:} (\l c\r \times \l \rho_0\r{:} \l \tau_0\r),$$ where $[a_1, \rho_0]=[a_2, \rho_0]=[a_2,c]=1$, $a_1^c=a_1^2$, $a_1^{\tau_0}=a_1^{-1}$ and $a_2^{\tau_0}=a_2^{-1}$. Furthermore, set $x=a_1a_2c\rho_0$ and $y=\tau_0$. Then the bi-rotary map $\Map(X,x,y)$ has Euler characteristic 
 \[\chi =24m_2 ( \frac{7}{12m_2} +\frac{1}{4m_2}-\frac{1}{2})=20-12m_2=-2^{42f}. \]  
  Moreover, we have $\O_{2}(X)=\l \rho_0\r \cong \Z_2 $ and the Hall $\{2,3\}$-subgroup of $X/\O_{2}(X)$ is isomorphic to  $\l \bar{x}\r\times \l \bar{y}\r \cong \Z_3\times \Z_2$. 
 This gives examples of Line~2 in Table~\ref{table solvable}. \qed
\end{example}



\begin{example}
 For any positive integer $f$, let $k_2=(2^{20f-10}+1)/41$. 
 Then $k_2$ is an integer 
  and let $X=\Z_{k_2}{\times}\D_{168}=\l b\r{\times}(\l g\r{:}\l d\r)$. Set $x=bg$ and $ y=d$. Then the  bi-rotary map $\Map(X,x, y)$ is of  Euler characteristic
    \[\chi=168k_2(\frac{1}{84k_2}-\frac{1}{2}+\frac{1}{84})=2-82k_2=-2^{20f-9} .\]
      Moreover, we have $\O_{2}(X)=\l g^{21}\r\cong \Z_4  $ and the Hall $\{2,3\}$-subgroup of $X/\O_{2}(X)$ is isomorphic to  $\D_6$. 
 This gives examples of Line~3 in Table~\ref{table solvable}. \qed
\end{example}

\begin{example}
 For any positive integer $f$,
  let $m=(11^{54f-45}+55)/3^3$. 
  Then 
  $m/66$ is  coprime to $66$. 
 Let $X=\Z_{55}{\times}\D_{m}=\l g_1\r{\times}(\l g_2\r{:}\l g_3\r)$. 
 Set $x=g_1g_2g_3$ and $ y=g_3$. Then the  bi-rotary map $\Map(X,x, y)$ is of  Euler characteristic
 \[
 \chi= 55m(\frac{1}{2{\cdot}55}-\frac{1}{2}+\frac{1}{m})=55-3^3m=-11^{54f-45}. 
 \]
  Moreover, we have $\O_{11}(X)=\l g_1^{5},g_2^{m/22} \r\cong \Z^2_{11}  $ and the Hall $\{2,3\}$-subgroup of $X/\O_{11}(X)$ is isomorphic to  $\D_6$. 
 This gives examples of Line~4 in Table~\ref{table solvable}.
\qed \end{example}

\begin{example}
 For any positive integer $f$ with $f\not\equiv 3 \pmod{11}$,
  let $m_2=(2^{290f-81}+5)/59$. Then $m_2/11$ is an integer which is coprime to $2{\cdot}3{\cdot}5{\cdot}11$. Now, set 
  \[X=(\Z_{11}\times\Z_{\frac{m_2}{11}}){:}(\Z_5{\times}(\Z_3{:}\D_{16}))=(\l a_1\r \times \l a_2\r){:}(\l b\r{\times}(\l c\r{:}(\l x_0\r{:}\l  y_0\r))),\]
   such that 
    \[ (a_1a_2)^{b}=a_1^4a_2, \ (a_1a_2)^{c}=a_1a_2, \ (a_1a_2c)^{x_0}=a_1a_2^{-1}c, \ (a_1a_2c)^{ y_0}=a_1^{-1}a_2^{-1}c^{-1}.  \]
 Set $x=a_1a_2bcx_0$ and $ y= y_0$. Then the  bi-rotary map $\Map(G,x, y)$ is of  Euler characteristic
 \[ \chi= 240m_2(\frac{1}{120}-\frac{1}{2}+\frac{1}{24m_2})=10-118m_2= -2^{290f-80}.\]
   Moreover, we have $\O_{2}(X)=\l x_0^2\r\cong \Z_{4}  $ and the Hall $\{2,3\}$-subgroup of $X/\O_{2}(X)$ is isomorphic to  $\Z_2\times \D_6$. 
 This gives examples of Line~5 in Table~\ref{table solvable}.
\qed \end{example}

\begin{example}  
 For any positive integer $f$ with $f\not\equiv 0 \pmod{3}$ and  $f\not\equiv 252 \pmod{421}$,
  let $m_2=(2^{1260f-1192}+5)/3^3$. Then $m_2/421$ is an integer which is coprime to $2{\cdot}3{\cdot}5{\cdot}421$. Now, set 
  \[X=(\Z_{421}\times \Z_{\frac{m_2}{421}}){:}(\Z_5{\times}(\Z_3{:}\D_{8}))=(\l a_1\r\times \l a_2\r){:}(\l b\r{\times}(\l c\r{:}(\l x_0\r{:}\l  y_0\r))),\]
   such that 
    \[ (a_1a_2)^{b}=a_1^{252}a_2, \ (a_1a_2)^{c}=a_1a_2, \ (a_1a_2c)^{x_0}=a_1a_2^{-1}c^{-1}, \ (a_1a_2c)^{ y_0}=a_1^{-1}a_2^{-1}c^{-1}.  \]
 Set $x=abcx_0$ and $ y= y_0$. Then the  bi-rotary map $\Map(X,x, y)$ is of  Euler characteristic
 \[ \chi= 120m_2(\frac{1}{20}-\frac{1}{2}+\frac{1}{12 m_2})=10-54m_2= -2^{1260f-1191}.\]
  Moreover, we have $\O_{2}(X)=\l x_0^2\r\cong \Z_{2}  $ and the Hall $\{2,3\}$-subgroup of $X/\O_{2}(X)$ is isomorphic to  $\Z_2\times \D_6$. 
 This gives examples of Line~6 in Table~\ref{table solvable}.
\qed \end{example}

\begin{example}
For any positive integer $f$ with $f\not\equiv 0\pmod{3}$ and $f\not\equiv 4\pmod{5}$, set $k=11^{2f-1}+4$. Then $k/15$ is an integer which is coprime to $2{\cdot}3{\cdot}5{\cdot}11$.
Let $H=\l x,y\r$ be a subgroup of $\GL(3,11)$, where
\[x_0=\begin{pmatrix}
 0&0&-2\\
 -2&0&0\\
 0&-2&0\\
\end{pmatrix} \ \mbox{and }\ y_0=\begin{pmatrix}
 1&&\\
 &-1&\\
 &&-1\\
\end{pmatrix}.\]
Let $U=\l u_1,u_2,u_3\r \cong \Z_{11}^3$ and let $X=\Z_{k/15}{\times} (U{:}H)$, where $H$ acts on $U$ naturally and $\Z_{k/15}=\l b\r$. Then $X\cong \Z_{k/15}{\times}\Z_{11}^3{:}(\A_4{\times }\Z_5)$. Now set $x=bu_1x_0$, $y=y_0$ and let $\Map(X,x,y)$ be a bi-rotary map of Euler characteristic
\[\chi=2^2{\cdot} 11^3{\cdot}k(\frac{1}{k}-\frac{1}{2}+\frac{1}{4})=4\cdot 11^3-11^3\cdot k=-11^{2f+2}.\]
  Moreover, we have $\O_{11}(X)=U\cong \Z_{11}^3  $ and the Hall $\{2,3\}$-subgroup of $X/\O_{2}(X)$ is isomorphic to  $\A_4$. 
 This gives examples of Line~7 in Table~\ref{table solvable}.
\qed \end{example}

\subsection{\texorpdfstring{Examples of bi-rotary maps with $\chi=-p^n$: non-solvable cases}{Examples of non-solvable cases }}
In this subsection, we present examples for the four cases presented in Theorem~\ref{thm:main-non-solvable}.


 \begin{example} 
Suppose $X=\Z_7^3.\PSL(2,7)$ where $\PSL(2,7)\cong \Omega(3,7)$ acts irreducibly on $\Z_7^3$ and this extension is non-split. 
With the aid of Magma\cite{MAGMA}, there is a unique non-split group $X$ and $X$ has presentation $X=\l x, y\mid x^3, y^2,[ y,x]^4,R\r$ for some exceptional relations $R$. Hence $\Map(X,x, y)$ is a bi-rotary of characteristic 
\[\chi=2^3{\cdot} 3{\cdot} 7^4(\frac{1}{3}-\frac{1}{2}+\frac{1}{8})=-7^4.\]
This gives an example of Theorem~\ref{thm:main-non-solvable} {\rm (i)}.
\qed \end{example}

\begin{example}
    Let $X=\PSL(2,7)\times \Z_{\ell}$ where $\ell=(7^6+8)/9$.
    By \cite{dazevedoBirotaryMapsNegative2019}, $\PSL(2,7)$ has
    a rotary pair $(x, y)$ such that $|x|=3$ and $|\omega|=4$.
    Then $X$ has a rotary pair $(x_1, y_1)$ such that $|x_1|=3\ell$ and $|\omega|=4$.
    Hence $\Map(X,x_1, y_1)$ is  a bi-rotary of characteristic
    \[\chi=2^3{\cdot} 3{\cdot} 7{\cdot} \ell(\frac{1}{3\ell}-\frac{1}{2}+\frac{1}{8})=-7^7.\]
    This gives an example of Theorem~\ref{thm:main-non-solvable} {\rm (i)}.
\qed \end{example}

\begin{example}
Suppose $X=A_5$. Let $x=(1\ 3\ 5)$ and $ y=(1\ 2)(3\ 4)$. Then $\Map(X,x, y)$ is a bi-rotary of characteristic $\chi=-4$.
This gives an example of Theorem~\ref{thm:main-non-solvable} {\rm (ii)}.
\qed \end{example}

\begin{example}
    Let $P=\l x_2\r{:}\l  y_2\r\cong \D_{16}$ be a dihedral group, and let $M$ be the maximal subgroup of $X$ which contains $x_2 y_2$. Let $\varphi: P\rightarrow \Aut(\Z_5)$ be the homomorphism from $P$ to $\Aut(\Z_5)$ with kernel $M$. Now
    let $X=\PSL(2,8)\times \Z_{k'}\times (\Z_5{{:}_\varphi} P)$, where $k'=(2^{443}+45)/(7{\cdot} 179)$.
    Moreover, let $a,b$ be the generators of $\Z_5$ and $\Z_{k'}$ respectively. Note that, there exists generating pair $x_1, y_1$ of $\PSL(2,8)$ with $|x_1|=7,$ $| y_1|=2$ and $| y_1 y_1^{x_1}|=9$. Set $x= x_1abx_2$ and $ y= y_1 y_2$. 
    Then $|x|=56k'$, $| y y^x|=180$ and $\l x, y\r =G$.
    This gives a bi-rotary map $\Map(X,x, y)$  of 
 Euler characteristic
    \[\chi=2^7{\cdot }5{\cdot} 7{\cdot} 3^2{\cdot}k'(\frac{1}{7{\cdot} 2^3{\cdot} k'}-\frac{1}{2}+\frac{1}{2^3{\cdot} 5{\cdot} 3^2})=-2^{447}.\]
    This gives an example of Theorem~\ref{thm:main-non-solvable} {\rm (ii)}.
\qed \end{example}

\begin{example}
	Let $x_1=(1\ 2\ 3)$, $ y_1=(1\ 4)(3\ 5)$ be two elements in $\A_5$. We have 
	$|x_1|=3$, $| y_1|=2$, $| y_1 y_1^{x_1}|=5$ and $\l x_1,  y_1\r=\A_5$.
    Let $P=(\l u_1\r {\times} \l u_2\r){:}\l u_3\r\cong \Z_2^2{:}\Z_4$, where $u_1^{u_3}=u_2$ and $u_2^{u_3}=u_1$. Set $x_2=u_3$ and $ y_2=u_1$. It follows $|x_2|=4$ and $| y_2|=| y_2 y_2^{x_2}|=2$. Let $m'=(2^{293}+33)/(25{\cdot} 13)$. Note that $m'=23{\cdot}p_1{\cdot}p_2$, where $p_1<p_2$ are two primes and $4\mid p_1-1$. Suppose that $\l a_0\r{\times}\l a_1\r{\times}\l a_2\r\cong \Z_{m'}$ and $\l b\r\cong \Z_{11}$, where $|a_0|=23$, $|a_1|=p_1$ and $|a_2|=p_2$. 
    Now, set 
    \[X=\A_5{\times}(\Z_{m'}{:}(\Z_{11}{\times}P))=\A_5{\times}(\l a_0\r{\times}\l a_1\r{\times}\l a_2\r){:}(\l b\r{\times}\l x_2,  y_2\r ),\]
    such that 
    \[(a_0a_1a_2)^{b}=a_0^{2}a_1a_2, \ (a_0a_1a_2)^{ y_2}=a_0^{-1}a_1^{-1}a_2^{-1}, \ (a_0a_1a_2)^{x_2}=a_0a_1^{e}a_2, \]
    where $e=1541127$.
    Moreover, set $x=x_1 a_0a_1a_2bx_2$ and $ y= y_1 y_2$. 
    Then $|x|=132=2^2{\cdot}3{\cdot}11$, $| y y^x|=10m'$ and $\l x, y\r =X$.
    This gives a bi-rotary map $\Map(X,x, y)$  of 
 Euler characteristic
    \[\chi=60{\cdot}2^4{\cdot }11{\cdot}m'(\frac{1}{2^2{\cdot}3{\cdot}11}-\frac{1}{2}+\frac{1}{20m'})=-2^{297}.\]
    This gives an example of Theorem~\ref{thm:main-non-solvable} {\rm (ii)}.
\qed \end{example}

\begin{example}
    Let $X_1=\PGL(2,7)$ and $X_2=\Z_{13}\times \Z_{m'}{:}\operatorname{D}_8=\l b\r\times (\l a\r{:}(\l x_2\r{:}\l  y_2\r))$, where $|a|=m'=\frac{2^{89}+7\cdot 13}{3\cdot 181}$, $|b|=13$ and $a^{x_2}=a^{ y_2}=a^{-1}$.
Using Magma, one can find that there exist $x_1,  y_1\in X_1$, such that $|x_1|=7, | y_1|=2$, $|[x_1,  y_1]|=3$ and $\l x_1,  y_1\r=X_1$. Now set $X_0=X_1\times X_2$,  $x=x_1 abx_2$, $y= y_1 y_2$ and $X=\l x, y\r \leq X_0$. 
Then $|x|=2^2{\cdot}7{\cdot}13$, $|[x,y]|=2\cdot 3\cdot m'$, and $X$ is a subgroup of index $2$ in $X_0$ with ${\rm O}_2(X)=\l x_2^2\r$.
Moreover $X/({\rm O}_2(X)\mathrm{soc}(X_1))\cong \operatorname{D}_{2m'}\times \Z_2$. Note that ${\rm O}_2(\operatorname{D}_{2m'}\times \Z_2)=\Z_2$. 
This gives a bi-rotary map $\Map(X,x,y)$ of Euler characteristic
    \[\chi=2^6{\cdot }3{\cdot} 7{\cdot}13{\cdot} m'(\frac{1}{2^2{\cdot}7{\cdot}  13}-\frac{1}{2}+\frac{1}{2^2{\cdot} 3{\cdot} m'})=-2^{93}.\]
This gives an example of Theorem~\ref{thm:main-non-solvable}~{\rm (iii)}.

\qed \end{example}

\begin{example}
    Let $X_1=\PGL(2,31)$ and $X_2=\Z_{m'}{:}\operatorname{D}_8=\langle a\r  {:}(\l x_2\r {:} \l  y_2\r)$, where $|a|=m'=\frac{2^{69}+15}{31\cdot 29}$ and $a^{x_2}=a^{ y_2}=a^{-1}$. By Magma, there exist $x_1,  y_1\in X_1\setminus \mathrm{soc}(X_1)$, such that $|x_1|=30, | y_1|=2$, $|[x_1,  y_1]|=31$ and $\l x_1,  y_1\r=X_1$. Now set $X_0=X_1\times X_2$, $x=x_1 ax_2$, $ y= y_1 y_2$ and $X=\l x,  y\r \leq X_0$. 
Then $|x|=2^2\cdot 3\cdot 5$, $|[x, y]|=2\cdot 31\cdot m'$, and $X$ is a subgroup of index $2$ in $X_0$ with ${\rm O}_2(X)=\l x_2^2, x_2 y_2\r$.
Moreover $X/({\rm O}_2(X)\mathrm{soc}(X_1))\cong \operatorname{D}_{2m'}$. Note that ${\rm O}_2(\operatorname{D}_{2m'})=1$. 
This gives a bi-rotary map $\Map(G,x, y)$ of Euler characteristic
\[\chi=2^8\cdot 3\cdot 5\cdot 31\cdot m'\left(\frac{1}{2^2\cdot 3\cdot 5}-\frac{1}{2}+\frac{1}{2^2\cdot 31\cdot m'}\right)
=-2^{75}\]
This gives an example of Theorem~\ref{thm:main-non-solvable} {\rm (iv)}.
\qed \end{example}

\vskip0.2in
\noindent\thanks{{\bf Acknowledgments.}
This work was supported by the National Natural Science Foundation of China (No. 12101518), the Fundamental Research Funds for the Central Universities (No. 20720210036, 20720240136).}

\bibliographystyle{abbrv}
\bibliography{bib}

\end{document}